\documentclass[reqno,12pt]{amsart}
\usepackage{amsmath,amssymb,amsthm,mathrsfs}
\usepackage{epsfig,color}
\usepackage[latin1]{inputenc}
\usepackage[english]{babel}
\usepackage[numbers]{natbib}
\usepackage[colorlinks, citecolor=blue, linkcolor=red]{hyperref}
\usepackage{dsfont}

\usepackage{hyperref}

\numberwithin{equation}{section}

\def\H{\mathbb{H}}

\newtheorem{ackn}{Acknowledgments\!}

\def\00{{\bf 0}}

\def\RR{\mathds R}

\def\ricc{\mathrm{Ric}}

\newcommand{\vol}{\mathrm{Vol}\,}

\newcommand{\eps}{{\varepsilon}}

\newtheorem*{theorem*}{Theorem}
\newtheorem{theorem}{Theorem}[section]
\newtheorem{lemma}[theorem]{Lemma}
\newtheorem{proposition}[theorem]{Proposition}

\newtheorem{corollary}[theorem]{Corollary}
\newtheorem{remark}[theorem]{Remark}

\begin{document}

 \title[Liouville Theorems on pseudohermitian manifolds]{Liouville Theorems on pseudohermitian manifolds with nonnegative Tanaka-Webster curvature}

  \date{}

\author{Giovanni Catino, Dario D. Monticelli, Alberto Roncoroni\\ and Xiaodong Wang}

\address{G. Catino, Dipartimento di Matematica, Politecnico di Milano, Piazza Leonardo da Vinci 32, 20133, Milano, Italy.}
\email{giovanni.catino@polimi.it}

\address{D. Monticelli, Dipartimento di Matematica, Politecnico di Milano, Piazza Leonardo da Vinci 32, 20133, Milano, Italy.}
\email{dario.monticelli@polimi.it}

\address{A. Roncoroni, Dipartimento di Matematica, Politecnico di Milano, Piazza Leonardo da Vinci 32, 20133, Milano, Italy.}
\email{alberto.roncoroni@polimi.it}

\address{X. Wang, Department of Mathematics, Michigan State University, East Lansing, Michigan, USA.}
\email{xwang@msu.edu}

\begin{abstract} In this paper we study positive solutions to the CR Yamabe equation in noncompact $(2n+1)$-dimensional Sasakian manifolds with nonnegative curvature. In particular, we show that the Heisenberg group $\mathbb{H}^1$ is the only (complete) Sasakian space with nonnegative Tanaka-Webster scalar curvature admitting a (nontrivial) positive solution. Moreover, under some natural assumptions, we prove this strong rigidity result in higher dimensions, extending the celebrated Jerison-Lee's result to curved manifolds. 
\end{abstract}

\maketitle

\begin{center}

\noindent{\it Key Words:} Sasakian manifolds, subelliptic semilinear equation, critical exponent, Liouville theorem, CR Yamabe problem, Folland-Stein-Sobolev inequality.

\bigskip

\noindent{\bf AMS subject classification:} 35J61, 32V20, 35B33, 53C25, 58J60

\end{center}

\

\section{Introduction}

\

The \emph{CR Yamabe equation} is the following subelliptic semilinear equation 
\begin{equation}\label{eq_Heisenberg}
-\Delta_{b} u =2 n^2 u^{\frac{n+2}{n}}\, \quad \text{in } \mathbb{H}^{n}\, ,
\end{equation}
where $\mathbb{H}^n$ is the \emph{Heisenberg group} 
\[
\mathbb{H}^{n}=\left\{  \xi=\left(  z,t\right)  \, : \, z\in\mathbb{C}^{n} \, ,t\in\mathbb{R}\right\}\, ,
\] 
and $\Delta_{b}u=u_{\alpha\bar{\alpha}}+u_{\bar\alpha\alpha}$ is the Heisenberg Laplacian (or sub-Lapacian). We refer to Section \ref{prel} for precise definitions and details. In \cite{JL}, Jersion and Lee obtained the following Liouville theorem: the only positive solution to \eqref{eq_Heisenberg} with \emph{finite energy}, i.e. 
$$
u\in L^{\frac{2n+2}{n}}(\mathbb{H}^n)
$$  
are the so-called \emph{Jerison-Lee bubbles} given by 
$$
\mathcal{U}_{\lambda,\mu}(z,t)=\frac{c_n}{\left| t+i|z|^2+z\cdot \mu +\lambda \right|^{n}}\, , $$ 
 for some $\lambda\in\mathbb{C}$, $\mu\in\mathbb{C}^n$ such that $\mathrm{Im}(\lambda)>\frac{\vert\mu\vert^2}{4}$ and for some explicit constant $c_n=c(n,\lambda)>0$. Actually, these functions corresponds to the extremals of the Folland-Stein-Sobolev inequality (see \cite{FS} and \cite{JL}) and are related to the CR Yamabe problem in the Heisenberg group (see e.g. \cite{ChCh,CMY,Gam1,Gam2,JL_CR1,JL,JL_CR2,W1} and the references therein). 
 
The Liouville theorem in \cite{JL} has been extended in various directions (see e.g. \cite{BP,GarVas} as well as \cite{W1,W2} on compact pseudohermitian manifolds with zero torsion). Very recently, in \cite{CLMR} the Jerison and Lee's result has been proved in $\mathbb{H}^1$ for \emph{all} positive solutions to \eqref{eq_Heisenberg} and in $\mathbb{H}^n$, with $n\geq 2$, under additional assumptions on the behavior of the solution at infinity (see \cite{CLMR} and \cite{flyvet}). 

Moreover, in the subcritical case, i.e. 
\begin{equation}\label{eq_Heisenberg}
-\Delta_{b} u =2 n^2 u^{q}\, \quad \text{in } \mathbb{H}^{n}\, ,
\end{equation}
with $1<q<\frac{n+2}{n}$ it is known that the only nonnegative solution is $u\equiv 0$ in $\mathbb{H}^n$ (see \cite{OuMa}).

\

Before presenting our main results we mention the Riemannian analogue of Liouville theorems for the following semilinear elliptic equation
\begin{equation}\label{eq_manifold}
\begin{cases}
-\Delta_g u=u^{q}\, \quad \text{in } M\, , \\
u\geq 0
\end{cases}
\end{equation}
where $(M^n,g)$ is a smooth, complete, noncompact, $n-$dimensional, with $n\geq 3$, Riemannian manifold with nonnegative  Ricci curvature, $\Delta_g$ denotes the Laplace-Beltrami operator and the exponent $q$ is subcritical, i.e. 
\begin{equation}\label{subcrit}
1<q< \dfrac{n+2}{n-2}\, ,
\end{equation}
or critical, i.e.
\begin{equation}\label{crit}
q=\dfrac{n+2}{n-2}\, . 
\end{equation}
We mention that the critical equation \eqref{eq_manifold} with \eqref{crit} arises in the well-known Yamabe problem of prescribing the scalar curvature of a conformal metric when the original one is scalar flat (see \cite{LeePar} for further details). 

In the Euclidean setting, problem \eqref{eq_manifold} is now well understood. In \cite{GS}, Gidas and Spruck showed that the only solutions to \eqref{eq_manifold} in the subcritical case \eqref{subcrit} is $u\equiv 0$ in $\mathbb{R}^n$. In \cite{CGS} (see also \cite{ChenLi,CLO,LiZhang}  and \cite{GNN,Obata} for previous important results), Caffarelli, Gidas and Spruck proved that any solution to \eqref{eq_manifold} in the critical case \eqref{crit} is radial and is given by 
\begin{equation}\label{AT_Eucl}
\mathcal V_{\lambda,x_0} (x) := \left( \frac{\lambda\sqrt{n(n-2)}}{\lambda^2 + |x-x_0|^2} \right)^{\frac{n-2}{2}} \,,
\end{equation}
for some $\lambda >0$ and $x_0\in\mathbb{R}^n$. The functions \eqref{AT_Eucl} were independently constructed in \cite{Aubin,Rod,Talenti} as minimizers of the Sobolev quotient 
$$
S_g(M):=\inf_{0\not\equiv u\in\mathcal{D}^{1,2}(M)}\frac{\int_{M} |\nabla u|^2\, dV_g}{\left(\int_{M} |u|^{\frac{2n}{n-2}}\, dV_g\right)^{\frac{n-2}{n}}}\, ,
$$
where $dV_g$ denotes the canonical volume element and 
$$
\mathcal{D}^{1,2}(M)=\left\{ u\in L^{\frac{2n}{n-2}}(M)\, : \, |\nabla u|\in L^2(M)\right\}\,, 
$$
with $M=\mathbb{R}^n$ (see \cite{Ron} for further details). For this reason the functions \eqref{AT_Eucl} are usually called {\em Aubin-Talenti bubbles}. Before passing to the Riemannian setting, we mention that similar Liouville theorems holds when the operator is the $p-$Laplacian, with $1<p<n$, and the exponent $\frac{n+2}{n-2}$ is replaced by $\frac{np}{n-p}-1$ (we refer to \cite{CMR,CFR,Ou,Sci,SZ,Vet,Vet_bis} for further details).

Extending such Liouville theorems for the critical equation to the case of more general Riemannian manifolds turns out to be a challenging and still open problem (in the subcritical case the result in \cite{GS} works also in any complete noncompact Riemannian manifold with nonnegative Ricci curvature). The conjecture is the following strong rigidity result: 

\

\noindent {\bf Conjecture:} {\em
let $(M^n,g)$ be a complete noncompact Riemannian manifold  with nonnegative Ricci curvature and let $u\in C^2(M)$ be a solution to 
\begin{equation}\label{eq_man_crit}
\begin{cases}
-\Delta_g u=u^{\frac{n+2}{n-2}}\, \quad \text{in } M\, , \\
u\geq 0\,,
\end{cases}
\end{equation}
then either $u\equiv 0$ in $M$ or $(M,g)$ is isometric to $\mathbb{R}^n$ with the Euclidean metric and $u$ is an Aubin-Talenti bubble.} 

\

This theorem is known to hold when the dimension of the Riemannian manifold is $n=3,4,5$ (see \cite{CaMo,CiFaPo} and also the already cited \cite{CGS,ChenLi,CLO,LiZhang,Ou,Vet_bis} for the Euclidean case) and holds under additional assumptions, such as $u\in\mathcal{D}^{1,2}(M)$ (or even weaker ones), in every dimension $n\geq 6$. We refer to the papers \cite{CaMo, CiFaPo, FMM, Ou, Vet_bis} and also for further details. Finally, the rigidity result still holds true on Cartan-Hadamard manifolds, assuming that the solution to \eqref{eq_man_crit} is either radially symmetric and $u\in\mathcal{D}^{1,2}(M)$ or a minimizer to the Sobolev quotient (see \cite{MurSoa} for further details).

\

In this paper,  we consider the critical subelliptic semilinear equation on complete Sasakian manifolds with nonnegative Tanaka-Webster Ricci curvature (see Section \ref{prel}). In particular, given a $(2n+1)$-dimensional complete Sasakian manifold $(M^{2n+1}, \theta, J, g)$ we consider the CR Yamabe equation
\begin{equation}\label{eqcritn_general}
-\Delta_b u = 2n^2 u^{\frac{n+2}{n}}\, \quad \text{in } M^{2n+1}\, ,
\end{equation}
where $\Delta_{b}u=u_{\alpha\bar{\alpha}}+u_{\bar\alpha\alpha}$ is the sub-Laplacian (see Section \ref{prel}).

\

Our main results are the following. The first is a characterization of the Heisenberg group $\mathbb{H}^1$ as the only (complete) three-dimensional Sasakian space with nonnegative Tanaka-Webster scalar curvature admitting a (nontrivial) positive solution to the CR Yamabe equation, without any further assumptions. This theorem generalizes previous results in \cite{CLMR, CaMo}. 
\begin{theorem}\label{teo1}
 Let $(M^{3}, \theta, J, g)$ be a $3$-dimensional complete Sasakian manifold with nonnegative Tanaka-Webster scalar curvature and let $u$ be a nonnegative solution to the following equation
\begin{equation}\label{eqcrit1}
-\Delta_b u = 2u^{3}\, \quad \text{in } M^3.
\end{equation}
Then either $u\equiv 0$ or $(M^{3}, \theta, J, g)$ is CR isometric to $\H^1$ with its standard structures and $u$ is a Jerison-Lee bubble, that is $$u \equiv \mathcal{U}_{\lambda,\mu},$$
 for some $\lambda,\mu\in\mathbb{C}$ such that $\mathrm{Im}(\lambda)>\frac{\vert\mu\vert^2}{4}$.
\end{theorem}

In the next results we deal with the higher dimensional case  where we prove similar strong rigidity results under additional assumptions.

Here and in the following we denote by $B_R$ the ball of radius $R$ centered at a fixed point with respect to the Carnot-Caratheodory distance (see Section \ref{prel}).

\begin{theorem}\label{teo2}
 Let $(M^{2n+1}, \theta, J, g)$, $n\geq 2$, be a $(2n+1)$-dimensional complete Sasakian manifold with the Tanaka-Webster curvature satisfying
$$
\ricc_H(X,X)\geq 0 \quad \text{for all horizontal vector field }\,X,
$$
and
$$
\mathrm{Vol}B_R\leq C R^{2n+2} \, ,
$$
for some $C>0$ and every $R>0$ large enough. Let $u\in L^{\frac{2n+2}{n}}(M)$ be a nonnegative solution to the following equation
\begin{equation}\label{eqcritn}
-\Delta_b u = 2n^2 u^{\frac{n+2}{n}}\, \quad \text{in } M^{2n+1}\, ,
\end{equation}
such that $u$ tends to zero at infinity.

Then either $u\equiv 0$ or $(M^{2n+1}, \theta, J, g)$ is CR isometric to $\H^n$ with its standard structures and $u$ is a Jerison-Lee bubble, that is $$u \equiv \mathcal{U}_{\lambda,\mu},$$
 for some $\lambda\in\mathbb{C},\mu\in\mathbb{C}^n$ such that $\mathrm{Im}(\lambda)>\frac{\vert\mu\vert^2}{4}$.
\end{theorem}

We actually prove the following result of which Theorem \ref{teo2} is an immediate consequence.

\begin{theorem}\label{teo3}
 Let $(M^{2n+1}, \theta, J, g)$, $n\geq 2$, be a $(2n+1)$-dimensional complete Sasakian manifold with  the Tanaka-Webster curvature satisfying
$$
\ricc_H(X,X)\geq 0 \quad \text{for all horizontal vector field }\,X,
$$
together with 
$$
\mathrm{Vol}B_R\leq C R^{2n+2} \, ,
$$
for some $C>0$ and every $R>0$ large enough. Let $u$ be a nonnegative solution to \eqref{eqcritn} in $M$ which tends to zero at infinity and such that
$$
\int_{A_R}u^{\frac{2(n+1)}{n}} \leq C R^{\sigma}\,, 
$$
for every $R>0$ large enough, where $A_R=B_R\setminus B_{R/2}$, and for some $\sigma<2$ if $n=2$, $\sigma=2$ if $n\geq 3$. Then either $u\equiv 0$ or $(M^{2n+1}, \theta, J, g)$ is CR isometric to $\H^n$ with its standard structures and $u$ is a Jerison-Lee bubble, that is $$u \equiv \mathcal{U}_{\lambda,\mu},$$
 for some $\lambda\in\mathbb{C},\mu\in\mathbb{C}^n$ such that $\mathrm{Im}(\lambda)>\frac{\vert\mu\vert^2}{4}$.
\end{theorem}

\begin{remark}
We point out that in the Sasakian context a Bishop-Gromov volume comparison is still an open and challenging question for $n\geq 2$; the case $n=1$ was treated in \cite{agralee} (see Theorem \ref{t-comp} $(1)$ for further details). It is well-known (see e.g. \cite{leeli}) that given $(M^{2n+1}, \theta, J, g)$, $n\geq 2$, a $(2n+1)$-dimensional complete Sasakian manifold such that
\begin{equation}\label{foto}
\ricc_H(X,X)\geq \mathrm{R}(X,JX, X, JX)\geq 0 \quad \text{for all horizontal vector field }\,X,
\end{equation}
then 
$$
\mathrm{Vol}B_R\leq C R^{2n+2} \, ,
$$
for some $C>0$ and every $R>0$ large enough (see Theorem \ref{t-comp} $(2)$). In particular, we have that Theorems \ref{teo2} and \ref{teo3} hold replacing the assumptions on $\mathrm{Ric}_H$ and on $\mathrm{Vol}B_R$ with \eqref{foto}.
\end{remark} 

Important corollaries of our theorems are the following. We define the Folland-Stein-Sobolev quotient as
$$
\mathcal{S}_\theta(M)=\inf_{0\not\equiv u\in\mathcal{D}^{1,2}(M)}\frac{\int_M |\partial u|^2}{\left(\int_M |u|^{\frac{2n+2}{n}}\right)^{\frac{n}{n+1}}}\, ,
$$
where
$$
\mathcal{D}^{1,2}(M)=\left\{ u\in L^{\frac{2n+2}{n}}(M) \, :\ \, |\partial u|\in L^2(M) \right\}\,.
$$
We refer to Section \ref{prel} for the definition of $\partial u$. It is well known the validity of the Folland-Stein-Sobolev inequality on a complete Sasakian manifold $M$, i.e. $\mathcal{S}_\theta(M)>0$, is equivalent to an optimal volume growth of balls from below (see \cite[Theorem 2.6]{BKim}). Under this assumption, a positive finite energy solution to \eqref{eqcritn} vanishes at infinity (see Lemma \ref{Moser} below).

\begin{corollary}\label{Sob_gen}
Let $(M^{2n+1}, \theta, J, g)$, $n\geq 2$, be a $(2n+1)$-dimensional complete Sasakian manifold with the Tanaka-Webster curvature satisfying
$$
\ricc_H(X,X)\geq 0 \quad \text{for all horizontal vector field }\,X,
$$
and
$$
cR^{2n+2}\leq \mathrm{Vol}B_R\leq C R^{2n+2} \, ,
$$
for some $c,C>0$ and every $R>0$ large enough. Let $u\in L^{\frac{2n+2}{n}}(M)$ be a nonnegative solution to  \eqref{eqcritn}. Then either $u\equiv 0$ or $(M^{2n+1}, \theta, J, g)$ is CR isometric to $\H^n$ with its standard structures and $u$ is a Jerison-Lee bubble, that is $$u \equiv \mathcal{U}_{\lambda,\mu},$$
 for some $\lambda\in\mathbb{C},\mu\in\mathbb{C}^n$ such that $\mathrm{Im}(\lambda)>\frac{\vert\mu\vert^2}{4}$.
\end{corollary}

This result recovers and improves the seminal one by Jerison-Lee \cite{JL} for finite energy solutions in the Heisenberg group to the curved setting. Moreover, for Folland-Stein-Sobolev minimizers we have the following.

\begin{corollary}\label{Sob}
Let $(M^{2n+1}, \theta, J, g)$ be a $(2n+1)$-dimensional complete, nonflat Sasakian manifold. If
\begin{itemize}
\item[i)] $n=1$ and the Tanaka-Webster scalar curvature is nonnegative, or
\item[ii)] $n\geq 2$, the Tanaka-Webster curvature satisfies
$$
\ricc_H(X,X)\geq 0 \quad \text{for all horizontal vector field }\,X,
$$
and
$$
\mathrm{Vol}B_R\leq C R^{2n+2} \, ,
$$
for some $C>0$ and every $R>0$ large enough,
\end{itemize}
then there exists no $0<u\in\mathcal{D}^{1,2}(M)$ which minimizes $\mathcal{S}_\theta(M)$.
\end{corollary}
In particular, if $n\geq 2$,
$$\ricc_H(X,X)=0\quad\text{for all horizontal vector field }\,X\quad\text{and}\quad \mathrm{Vol}B_R\leq C R^{2n+2},$$ then there exists no $0<u\in\mathcal{D}^{1,2}(M)$ which minimizes the CR-Yamabe quotient:
$$
Y_\theta(M):=\inf_{0\not\equiv u\in C^{\infty}_0(M)}\frac{\int_M \frac{n+2}{n}|\partial u|^2+R u^2}{\left(\int_M |u|^{\frac{2n+2}{n}}\right)^{\frac{n}{n+1}}}\, .
$$

\

The proof of our results relies on a generalization of the remarkable differential identity proved in \cite{JL} (see also \cite{OuMa}) to the Sasakian setting (see e.g. \cite{W1,W2,OuXuMa})
, which involves a vector field depending on the solution $u$ and its derivatives and the Tanaka-Webster Ricci curvature of the manifold. From our curvature assumption we have that the divergence of this vector field is nonnegative. Inspired by \cite{CLMR} (see also \cite{CaMo,CMR}), through a test functions argument we are able to obtain integral estimates which imply that such divergence must vanish identically, thus giving the desired rigidity result for both the solution and the manifold $M^{2n+1}$. When $n=1$, our argument goes through without any further assumption by using the sub-Laplacian comparison which yields to a lower-bound on the solution and an upper bound on the volume of balls (with respect to the Carnot-Caratheodory distance) \`a la Bishop-Gromov. While, when $n\geq 2$, to perform our argument we need to assume that the solution has finite energy (or actually a weaker energy assumption) and goes to zero at infinity. This is reminiscent of the papers \cite{CLMR, flyvet}.

\

\bigskip

\noindent{\bf Organization of the paper.} The paper is organized as follows: in Section \ref{prel} we collect some preliminaries regarding CR and Sasakian geometry, then we present the main differential identity together with an integral inequality which will be the starting point of our approach. In Section \ref{int_est} we prove several integral estimates which are used in Section \ref{proofs} to prove the main results stated above.

\

\section{Preliminaries}\label{prel}

\subsection{Notations.} We first give a brief introduction to psedohermitian geometry and fix our notations, following the standard reference \cite{Lee}. A \emph{CR manifold} is a smooth manifold $M$ endowed with a distinguished subbundle $T^{1,0}$ of the complexified tangent bundle $\mathbb{C}TM=TM\otimes\mathbb{C}$ such that $[T^{1,0},T^{1,0}]\subseteq T^{1,0}$ (i.e. $T^{1,0}$ is formally integrable) and $T^{1,0}\cap T^{0,1}=\{0\}$, where $T^{0,1}=\overline{T^{1,0}}$ (i.e. $T^{1,0}$ is almost Lagrangian). The bundle $T^{1,0}$ is called a \emph{CR structure} on the manifold $M$. A CR manifold is said to be of \emph{hypersurface type} if $\mathrm{dim}_\mathbb{R}M=2n+1$ and $\mathrm{dim}_{\mathbb{C}}T^{1,0}=n$. Then $H(M):=\{X+\overline{X}:X\in T^{1,0}\} $ is a  real subbundle of $TM$ of rank $2n$ and  called the horizontal distribution. On $H(M)$ there is an almost complex structure $J$ defined by $J\left( X+\overline{X}\right)=i \left( X-\overline{X}\right)$.

We will consider only CR manifolds $M$ which are of hypersurface type and oriented. Then it is possible to associate to its CR structure $T^{1,0}$ a one form $\theta$ globally defined on $M$ such that $\mathrm{ker}(\theta)=H(M) $, which is unique modulo a multiple of nonzero function on $M$: a choice of a nonzero multiple of such $\theta$ is called a \emph{pseudohermitian structure} on $M$. The \emph{Levi form} of $\theta$ is defined by  
$$
L_{\theta}(V,W):=- i d\theta (V,\overline{W}) \,  , \quad \text{for all } V,W\in T^{1,0}\, . 
$$
We say that the CR structure is \emph{strictly pseudoconvex} if the Levi form $L_\theta$ is positive definite for some choice of $\theta$. We simply call the triple $(M^{2n+1}, \theta,J)$ a \emph{pseudohermitian manifold.} In this case, $\theta$ is a contact form and $G_{\theta}:=d\theta(\cdot, J\cdot)$ is a $J$-invariant positive-definite bilinear form on $H(M)$, i.e. it defines a $J$-invariant Riemannian metric on $H(M)$. 

\noindent The \emph{Reeb vector field} $T$ is defined by 
$$
\theta(T)=1 \quad \text{and} \quad  d\theta(T,X)=0 \quad \text{ for all } X\in TM\, ,
$$
then $TM=H(M)\oplus  \mathrm{Span}(T)$. We
have a natural Riemannian metric $g_{\theta}$ on $M$ such that this decomposition is orthogonal, 
$g_{\theta}\left( T,T\right) =1$ and its restriction on $H(M)$ equals $G_{\theta}$. The \emph{volume form} of $g_{\theta}$ is 
$$
dV_\theta=\theta\wedge(d\theta)^n/n!\, , 
$$ 
and is often omitted when we integrate. 

\noindent As usual, we extend $g_{\theta}$ as a complex bilinear form on the complexified tangent bundle
$\mathbb{C}TM=T^{1,0}\oplus T^{0,1}\oplus\mathrm{Span}(T)$. Throughout the paper, we work with the \emph{Tanaka-Webster connection} $\nabla$ which  is compatible with the metric  $g_{\theta }$, but it has a nontrivial torsion. The torsion $\tau $ satisfies
$$
\tau \left( Z,W\right) =0\, , \quad  \tau \left( Z,\overline{W}\right) =\omega \left( Z,\overline{W}\right) T\, ,
\quad \tau \left( T,J\cdot \right) =-J\tau \left( T,\cdot \right)
$$
for any $Z,W\in T^{1,0}\left( M\right) $.
We define $A:T\left( M\right) \rightarrow T\left( M\right) $ by 
$$
AX=\tau
\left( T,X\right) \, . 
$$ 
It is customary to simply call $A$ the \emph{torsion of the CR manifold}. It is easy to see that $A$ is symmetric. Moreover 
$$
AT=0\, , \quad AH\left( M\right) \subset
H\left( M\right)\, , \quad AJX=-JAX\, .
$$ 
It is well known that
$$
A=0\quad  \Longleftrightarrow \quad g_{\theta}\,  \text{ is \emph{Sasakian}, i.e. the Reeb vector field $T$ is \emph{Killing.}}
$$
Therefore Sasakian manifolds are precisely the strictly pseudoconvex CR manifolds with a fixed contact form $\theta$ such that the Levi form $L_{\theta}$ is positive definite and the torsion of the Tanaka-Webster connection vanishes.

On a pseudohermitian manifold $(M^{2n+1},\theta,J,g_{\theta})$, a Lipschitz \emph{curve} $\gamma :\left[ 0,1\right]
\rightarrow M$ is called \emph{admissible} if 
$$
\gamma ^{\prime }\left( t\right)
\in H_{\gamma \left( t\right) }\left( M\right) \, ,\quad \text{ for all $t$ where $\gamma'$ is defined}.
$$ 
The \emph{Carnot-Caratheodory distance} is defined as%
\begin{equation*}
d\left( p,q\right) =\inf \left\{ L\left( \gamma \right) :\gamma :\left[ 0,1%
\right] \rightarrow M\text{ is an admissible curve with }\gamma(0)=p \text{ and } \gamma(1)=q\right\} ,
\end{equation*}
where $p,q\in M$ and 
$$
L(\gamma)=\int_0^1 \sqrt{g_\theta(\gamma'(t),\gamma'(t))}\, dt\, , 
$$
is the \emph{length} of the curve $\gamma$.

\noindent The Carnot-Caratheodory distance is indeed a distance function that induces the same topology on $(M^{2n+1},\theta,J,g_{\theta})$. One can also define geodesics, the exponential map, cut points etc. in this
sub-riemannian setting. We refer to \cite{St} and \cite{book} and references therein for details.

To do calculations, it is convenient to work with a local frame $\{ Z_\alpha\}_{\alpha=1}^n$ for  $T^{1,0}$ such that
$$
d\theta (Z_\alpha, \bar{Z_\beta})=2\delta_{\alpha\bar{\beta}}\, .
$$ 
Let $\{\theta^\alpha\}_{\alpha=1}^n$ be the dual basis. We define, for all $\beta=1,\cdots,n$
$$
Z_{\bar{\beta}}:=\bar{Z}_{\beta}
\quad 	\text{ and } \quad \theta_{\bar{\beta}}:=\bar{\theta}_{\beta} \, .
$$
Therefore
$$
d\theta= 2i \delta_{\alpha\bar{\beta}}\theta^\alpha\wedge\theta^{\bar{\beta}}\, . 
$$

\

Throughout the paper we adopt Einstein summation convention over repeated indices. 

\

Moreover, given a Sasakian manifold $(M^{2n+1},\theta,J,g_{\theta})$ and a smooth function $f:M\rightarrow\mathbb{R}$ we denote its derivatives by
$$
f_\alpha=Z_\alpha f \, ,\quad f_{\bar{\alpha}}=Z_{\bar{\alpha}} f\, , \quad f_0=Tf \, , 
\quad f_{\alpha\bar{\beta}}=Z_{\bar{\beta}}\left( Z_\alpha f\right)-\nabla_{Z_{\bar{\beta}}}Z_\alpha f \, ,  \quad f_{0\alpha}=Z_{\alpha}Tf\, ,
$$
and so on. In doing calculations, we freely use the following commutation rules that can be found in  \cite{Lee}
$$
f_{\alpha\beta}-f_{\beta\alpha}=0 \quad \text{ and } \quad f_{\alpha\bar\beta}-f_{\bar\beta\alpha}=2i\delta_{\alpha\bar{\beta}}f_0 \, ,
$$
in addition, since the torsion of the Tanaka-Webster connection vanishes, we have
$$
f_{0\alpha}-f_{\alpha 0}=0\, \quad \text{ and } \quad f_{\alpha\beta 0}-f_{\alpha 0\beta}=f_{\alpha 0\bar\beta }-f_{\alpha \bar\beta 0}=0\, . 
$$
Moreover, 
$$
f_{\alpha\beta\bar{\gamma}}-f_{\alpha\bar{\gamma}\beta}=2i \delta_{\beta\bar{\gamma}}f_{\alpha 0} + R_{\bar\mu\alpha\beta\bar{\gamma}}f_{\mu}\, ,
$$
where $R_{\alpha\bar\mu\beta\bar{\gamma}}$ denotes the Tanaka-Webster Riemann curvature, $R_{\alpha\bar\beta}=R_{\mu\alpha\bar\mu\bar{\beta}}$ denotes the Tanaka-Webster Ricci curvature and $R=R_{\alpha\bar{\alpha}}$ denotes the Tanaka-Webster scalar curvature.

Finally, we define
$$
\vert\partial f\vert^2:= f_\alpha f_{\bar{\alpha}} 
$$
and 
$$
\Delta_bf:=f_{\alpha\bar{\alpha}}+ f_{\bar{\alpha}\alpha} \, ,
$$
which is the so-called \emph{sub-Laplacian.} 

We recall the Heisenberg group $\mathbb{H}^n$ is the flat model in the Sasakian setting.
We have
$$
\mathbb{H}^n:=\mathbb{C}^n\times\mathbb{R}
$$
with coordinates $\xi=(z,t)=(z_1,\dots,z_n,t)\in\mathbb{H}^n$ and with the group law $\circ$: given $\xi=(z,t)$ and $\zeta=(w,s)$
$$
(z,t)\circ(w,s)=\left(z+w,t+s+2\mathrm{Im} (z^\alpha\bar{w}^\alpha)\right) \, .
$$
We define the following left-invariant (with respect to $\circ$) vector fields in $\mathbb{H}^n$
$$
Z_\alpha=\frac{\partial}{\partial z^\alpha}+i\bar{z}_\alpha\frac{\partial}{\partial t} \quad \text{and} \quad Z_{\bar{\alpha}}=\frac{\partial}{\partial \bar{z}^\alpha}-i z_\alpha\frac{\partial}{\partial t} \quad \text{for } \alpha=1,\dots,n.
$$
The standard CR structure is given by the bundle $T^{1,0}$ spanned by the vector fields $Z_\alpha$, for $\alpha=1,\ldots,n$, and  the standard contact form is 
$$
\Theta=dt+\sum_{\alpha=1}^niz^\alpha d\bar{z}^\alpha-i\bar{z}^\alpha dz^{\alpha}.
$$
The Reeb vector field is $T=\frac{\partial}{\partial t}$ and the Tanaka-Webster Riemann curvature and the torsion are zero.

\

\subsection{A differential identity and an integral inequality} Given $u>0$ a solution of \eqref{eqcritn_general} we consider the auxiliary function $f$ defined as follows
\begin{equation}\label{def_f}
e^f=u^{\frac{1}{n}} \, ,
\end{equation}
then $f$ solves
\begin{equation}\label{eq_f}
-\Delta	_{b} f= 2n \vert\partial f\vert^2 + 2ne^{2f}\quad \text{in } M^{2n+1}\, .
\end{equation}
We also introduce the function $g:M^{2n+1}\rightarrow\mathbb{C}$ such that
\begin{equation}\label{g}
g=\vert\partial f\vert^2 + e^{2f} -if_0 \, ,
\end{equation}
then \eqref{eq_f} can be rewritten as
\begin{equation}\label{eq_g}
f_{\alpha\bar\alpha}=-ng\quad 	\text{in } M^{2n+1}\, .
\end{equation}
As done in \cite{JL}  we define the following tensors
\begin{equation}\label{tens}
\begin{array}{lll}
& D_{\alpha\beta}=f_{\alpha\beta}-2f_\alpha f_\beta & D_\alpha=D_{\alpha\beta}f_{\bar{\beta}} \\
& E_{\alpha\bar\beta}=f_{\alpha\bar\beta}-\frac{1}{n}f_{\gamma\bar\gamma}\delta_{\alpha\bar\beta} \quad & E_\alpha=E_{\alpha\bar\beta}f_{\beta} \\
&G_\alpha=if_{0\alpha}-if_0f_\alpha+ e^{2f}f_\alpha+\vert\partial f\vert^2 f_{\alpha}\, .
\end{array}
\end{equation}
The above tensors, introduced in \cite{JL}, will be important in our argument. Moreover (see also \cite{OuXuMa}) we observe that
\begin{equation}\label{tensg}
\begin{array}{lll}
& E_{\alpha\bar\beta}=f_{\alpha\bar\beta}+g\delta_{\alpha\bar\beta}  & E_\alpha=f_{\alpha\bar\beta}f_\beta+gf_\alpha\\
&D_\alpha=f_{\alpha\beta}f_{\bar\beta}-2|\partial f|^2f_\alpha & G_\alpha=if_{0\alpha}+gf_\alpha\\
&\vert\partial f\vert^2_{\bar{\alpha}}=D_{\bar{\alpha}}+E_{\bar{\alpha}}+ \bar{g} f_{\bar{\alpha}}-2 f_{\bar{\alpha}}e^{2f} \\
&g_{\bar\alpha}=D_{\bar\alpha}+E_{\bar\alpha}+G_{\bar\alpha} & \bar g_\alpha=D_{\alpha}+E_{\alpha}+G_{\alpha}.
\end{array}
\end{equation}
We are now in a position to present the following differential identity which generalizes the one in \cite{JL} to the Sasakian setting (see e.g. \cite{W1,W2,OuXuMa})

\begin{proposition}\label{prop_JL} With the notations above, we have
$$
\mathcal{Q}=\mathrm{Re}Z_{\bar\alpha}\left\lbrace e^{2(n-1)f}\left[ \left( g+3if_0\right) E_\alpha + \left( g-if_0\right)D_\alpha -3i f_0 G_\alpha\right] \right\rbrace\, ,
$$
where
\begin{align}\label{defQ}
\mathcal{Q}=& e^{2nf}\left( \vert E_{\alpha\bar{\beta}}\vert^2 + \vert D_{\alpha\beta}\vert^2\right)   \\ \nonumber
\quad&+ e^{2(n-1)f}\left( \vert G_\alpha \vert^2 + \vert G_\alpha+ D_\alpha\vert^2 + \vert G_\alpha -E_\alpha\vert^2 + \vert D_{\alpha\beta}f_{\bar{\gamma}}+ E_{\alpha\bar{\gamma}}f_\beta\vert^2\right)\\\nonumber
\quad&+ e^{2(n-1)f}\left(n^2|\partial f|^2+e^{2f}\right)R_{\alpha\bar\beta}f_{\alpha}f_{\bar\beta}\,.
\end{align}
\end{proposition}

\begin{proof}
On the Heisenberg group, this is exactly formula (4.2) in \cite{JL}. The
above generalization can be proved along the same lines in \cite{W2} where a
modified version was established for the equation
\[
-\frac{4}{n^{2}}\Delta_{b}u+u=u^{\frac{n+2}{n}}
\]
on any pseudohermitian manifold with zero torsion. By scaling the pseudohermitian structure and $u$ properly, one can get an identity for a positive solution to $$
-\Delta_b u +\lambda u=u^{\frac{n+2}{n}}\, ,
$$
were $\lambda$ is any positive constant. This is explicitly worked out  in \cite[Remark 3.2]{OuXuMa}. Namely, the following Jerison-Lee identity holds for a positive solution to the above equation:
\begin{align*}
u^{\frac{2}{n}}&\mathrm{Re}Z_{\bar{\alpha}}\left\{ u^{-\frac{2}{n}}\left[ \left(\frac{|\partial u|^2}{u}+u^{\frac{n+2}{n}}+\lambda u\right) (\tilde D_\alpha+\tilde E_\alpha)-n i u_0(2\tilde  D_\alpha-2\tilde E_\alpha+3\tilde  G_\alpha)\right]\right\} \\
=& u^{-2} |\tilde D_{\alpha\beta}u_{\bar{\gamma}}+\tilde E_{\alpha\bar{\gamma}}u_{\beta}|^2+\frac{|\partial u|^2}{u^2}\mathscr{R}+\left(u^{\frac{2}{n}}+\lambda\right)\left( |\tilde  D_{\alpha\beta}|^2+|\tilde  E_{\alpha\bar{\beta}}|^2+\mathscr{R}\right) \\
&+ |\tilde G_\alpha+\tilde D_\alpha|^2 + |\tilde G_\alpha-\tilde E_\alpha|^2+|\tilde G_\alpha|^2\, ,
\end{align*}
where $\tilde{D},\tilde{E}, \tilde{G}$ are defined in \cite[Section 2]{OuXuMa}, $u$ solves 
$$
-\Delta_b u +\lambda u=u^{\frac{n+2}{n}}\, ,
$$
and 
$$
\mathscr{R}=R_{\alpha\bar{\beta}}u_{\alpha}u_{\bar{\beta}}-\frac{4(n+1)}{n^2}\lambda|\partial u|^2\, .
$$
From \eqref{def_f} it is easy to see that 
$$
\widetilde{D}_{\alpha}=n^2 e^{nf} D_\alpha\, , \quad \widetilde{E}_\alpha=n^2e^{nf}E_\alpha \, , \quad \widetilde{G}_\alpha=n^2 e^{nf}G_\alpha
$$
and 
$$
\widetilde{D}_{\alpha\beta}= n e^{nf}D_{\alpha\beta}\, , \quad \widetilde{E}_{\alpha\beta}=n e^{nf}E_{\alpha\beta} \, . 
$$
Eventually, by choosing $\lambda=0$ the conclusion follows after some manipulation.
\end{proof}

From this proposition we obtain the next lemma which extends a similar result in \cite{CLMR}.
\begin{lemma}\label{lemma0}
With the notations above, for every $n\geq 1$, every (real) nonnegative cut-off function $\eta$ with compact support, every $s>2$ and every $0\leq\beta<\frac{1}{12}$ we have
$$
\int_{M} \mathcal{Q}\Psi^{-\beta}\eta^s\leq  C \left(\int_{\text{supp}|\partial\eta|}\mathcal{Q}\Psi^{-\beta}\eta^s\right)^{1/2}\left(\int_{\text{supp}|\partial\eta|} e^{2(n-1)f}|g|^2\Psi^{-\beta}|\partial\eta|^2\eta^{s-2}\right)^{1/2}\,,
$$
for some $C>0$, where
\begin{equation}\label{defPsi}
\Psi:=g\bar{g}\, e^{-2f}=\vert g\vert^2 \, e^{-2f}\, .
\end{equation}
In particular
$$
\int_{M} \mathcal{Q}\Psi^{-\beta}\eta^s\leq  C \int_{M}e^{2(n-1)f}\vert g\vert^2 \Psi^{-\beta} \vert\partial\eta\vert^2 \eta^{s-2}\,.
$$
\end{lemma}

\begin{proof}
We define
$$
\mathcal{I}_1:=\int_{M} \mathcal{Q}\Psi^{-\beta}\eta^s\,.
$$
From Proposition \ref{prop_JL} we obtain
\begin{align}\label{eq0}
\int_{M} \mathcal{Q}\Psi^{-\beta}\eta^s=&\int_{M}  \mathrm{Re}Z_{\bar\alpha}\left\{e^{2(n-1)f} \left[ \left( g+3if_0\right) E_\alpha + \left( g-if_0\right)D_\alpha -3i f_0 G_\alpha\right]\right\}  \Psi^{-\beta} \eta^s \nonumber \\
=& \beta \int_{M}  \mathrm{Re}\left\lbrace e^{2(n-1)f}\left[ \left( g+3if_0\right) E_\alpha + \left( g-if_0\right)D_\alpha -3i f_0 G_\alpha\right]\Psi_{\bar{\alpha}}\right\rbrace \Psi^{-\beta-1} \eta^s \nonumber  \\
&-s \int_{M} \mathrm{Re}\left\lbrace e^{2(n-1)f}\left[ \left( g+3if_0\right) E_\alpha + \left( g-if_0\right)D_\alpha -3i f_0 G_\alpha\right]\eta_{\bar{\alpha}} \right\rbrace \Psi^{-\beta} \eta^{s-1}
\end{align}
where we integrate by parts. We now observe that, on the one hand, from the definition of $g$ \eqref{g} we have
\begin{align*}
\left( g+3if_0\right) E_\alpha + \left( g-if_0\right)D_\alpha -3i f_0 G_\alpha = &\,  \left( D_\alpha+ G_\alpha\right) (g-if_0) + \left( E_\alpha - G_\alpha\right) (g+3if_0)+ if_0G_\alpha \\
=& \left( D_\alpha+ G_\alpha\right) \left( \vert\partial f\vert^2 + e^{2f} -2i f_0\right)\\
&\, + \left( E_\alpha - G_\alpha\right) \left( \vert\partial f\vert^2 + e^{2f} +2if_0\right)+ if_0G_\alpha    \, ,
\end{align*}
and so, from Cauchy-Schwarz inequality
\begin{align*}
\vert \left( g+3if_0\right) E_\alpha + \left( g-if_0\right)D_\alpha -3i f_0 G_\alpha \vert\leq & \left\vert D_\alpha+ G_\alpha\right\vert \sqrt{\vert\partial f\vert^4 + e^{4f}+2\vert\partial f\vert^2 e^{2f} +4 f_0^2}\\
&\, + \left\vert E_\alpha - G_\alpha\right\vert \sqrt{\vert\partial f\vert^4 + e^{4f}+2\vert\partial f\vert^2 e^{2f} +4 f_0^2}\\
&\,+ \vert f_0\vert\vert G_\alpha\vert    \\
 \leq &\,  2\vert g\vert \left( \left\vert D_\alpha+ G_\alpha\right\vert +\left\vert E_\alpha - G_\alpha\right\vert  + \vert G_\alpha\vert \right)\\
 \leq &\, 2\sqrt{3}\vert g\vert\sqrt{\mathcal{Q}}e^{-(n-1)f}
\end{align*}
where we used the fact that
\begin{equation}\label{mod_g}
\vert g\vert=\sqrt{\vert\partial f\vert^4 + e^{4f}+2\vert\partial f\vert^2 e^{2f} + f_0^2}\, .
\end{equation}
Summing up, we have obtained the following
\begin{equation}\label{p1}
\vert \left( g+3if_0\right) E_\alpha + \left( g-if_0\right)D_\alpha -3i f_0 G_\alpha\vert \leq  \,  2\sqrt{3}\vert g\vert\sqrt{\mathcal{Q}}e^{-(n-1)f}\, .
\end{equation}
On the other hand, from \eqref{tensg} we have
\begin{align*}
\Psi_{\bar{\alpha}}=&\left[ \left( g\bar{g}\right) e^{-2f}\right]_{\bar{\alpha}}=e^{-2f} \left( \bar{g}g_{\bar{\alpha}} + g \bar{g}_{\bar{\alpha}}\right) - 2\left( g\bar{g}\right)f_{\bar{\alpha}}e^{-2f} \\
=& \,  e^{-2f} \left[ \bar{g}\left( D_{\bar\alpha}+E_{\bar\alpha}+G_{\bar\alpha}\right) + g \left( D_{\bar\alpha}+E_{\bar\alpha}-G_{\bar\alpha}+2\bar{g}f_{\bar{\alpha}}\right)\right] -2\left( g\bar{g}\right) f_{\bar{\alpha}}e^{-2f} \\
=\, & e^{-2f} \left[ D_{\bar{\alpha}} \left( g + \bar{g}\right) + E_{\bar{\alpha}}\left( g + \bar{g}\right) + G_{\bar{\alpha}}\left(\bar{g}-g \right) \right] \\
=&\,   e^{-2f} \left[ \left( D_{\bar{\alpha}} + G_{\bar{\alpha}}\right) \left( g + \bar{g}\right) + \left( E_{\bar{\alpha}}-G_{\bar{\alpha}}\right) \left( g + \bar{g}\right) + G_{\bar{\alpha}}\left(\bar{g}-g \right) \right] \\
=& \,   2 e^{-2f} \left[ \left( D_{\bar{\alpha}} + G_{\bar{\alpha}}\right) \left( \vert\partial f\vert^2 + e^{2f} \right) + \left( E_{\bar{\alpha}}-G_{\bar{\alpha}}\right) \left( \vert\partial f\vert^2 + e^{2f} \right) + i G_{\bar{\alpha}}f_0 \right] \, ,
\end{align*}
where we used the fact that
\begin{equation}\label{gbarra}
\bar{g}_{\bar{\alpha}}= D_{\bar\alpha}+E_{\bar\alpha}-G_{\bar\alpha}+2\bar{g}f_{\bar{\alpha}} \, .
\end{equation}
Indeed by \eqref{tensg} we have
$$
G_{\bar{\alpha}}=-if_{0\bar{\alpha}}+\bar{g}f_{\bar{\alpha}}\,,
$$
hence
\begin{align*}
  \bar{g}_{\bar{\alpha}}= & \left(|\partial f|^2+e^{2f}+if_0\right)_{\bar{\alpha}} \\
   =& D_{\bar{\alpha}}+E_{\bar{\alpha}}+ \bar{g} f_{\bar{\alpha}}+if_{0\bar{\alpha}}\\
   =& D_{\bar\alpha}+E_{\bar\alpha}-G_{\bar\alpha}+2\bar{g}f_{\bar{\alpha}} \, .
\end{align*}
Moreover,  from Cauchy-Schwarz inequality
\begin{align*}
\vert\Psi_{\bar{\alpha}}\vert\leq & \,   2 e^{-2f} \left[ \left\vert D_{\bar{\alpha}} + G_{\bar{\alpha}}\right\vert \sqrt{\vert\partial f\vert^4 + e^{4f}+2\vert\partial f\vert^2 e^{2f}} \right.\\
&\, \left.+ \left\vert E_{\bar{\alpha}}-G_{\bar{\alpha}}\right\vert \sqrt{\vert\partial f\vert^4 + e^{4f}+2\vert\partial f\vert^2 e^{2f}} +  \vert G_{\bar{\alpha}}\vert \vert f_0\vert \right] \\
\leq & 2 e^{-2f} \vert g\vert  \left[ \vert D_{\bar{\alpha}} + G_{\bar{\alpha}}\vert + \vert E_{\bar{\alpha}}-G_{\bar{\alpha}}\vert + \vert G_{\bar{\alpha}}\vert  \right]  \\
\leq & \,  2\sqrt{3} e^{-2f}\vert g\vert\sqrt{\mathcal{Q}}e^{-(n-1)f}\, ,
\end{align*}
i.e.
\begin{equation}\label{p2}
\vert \Psi_{\bar{\alpha}}\vert \leq  \,  2\sqrt{3} e^{-(n+1)f}\vert g\vert\sqrt{\mathcal{Q}}\, .
\end{equation}
Hence, by substituting \eqref{p1} and \eqref{p2} in \eqref{eq0} we get, after a Cauchy-Schwarz inequality,
\begin{align*}
\mathcal{I}_1= \left\vert \int_{M} \mathcal{Q}\Psi^{-\beta}\eta^s\right\vert \leq & 12\beta \int_{M} \mathcal{Q}\vert g\vert^2 e^{-2f}    \Psi^{-\beta-1} \eta^s \nonumber + 2\sqrt{3}s \int_{M}e^{(n-1)f} \vert g\vert \sqrt{\mathcal{Q}}  \Psi^{-\beta}\vert\partial \eta\vert \eta^{s-1}  \\
= & \,  12\beta \int_{M} \mathcal{Q}   \Psi^{-\beta} \eta^s + 2\sqrt{3}s \int_{M} e^{(n-1)f}\sqrt{\mathcal{Q}} \vert g\vert\Psi^{-\beta}\vert\partial \eta\vert \eta^{s-1}\,.
\end{align*}
Since $\beta<\frac{1}{12}$ we find
\begin{equation*}
  \mathcal{I}_1\leq C\int_{M}e^{(n-1)f} \sqrt{\mathcal{Q}} \vert g\vert\Psi^{-\beta}\vert\partial \eta\vert \eta^{s-1}\,.
\end{equation*}
We now use H\"{o}lder's inequality and we have
$$
\mathcal{I}_1\leq C \left(\int_{\mathrm{supp}|\partial\eta|} \mathcal{Q}\Psi^{-\beta}\eta^s\right)^\frac{1}{2}\left( \int_{\mathrm{supp}|\partial\eta|} e^{2(n-1)f}\vert g\vert^2 \Psi^{-\beta}\vert\partial\eta\vert^2\eta^{s-2} \right)^\frac{1}{2} \, ,
$$
for some $C>0$. Now the conclusion easily follows.
\end{proof}

\

\subsection{Comparison results and applications} We recall the following sub-Laplacian comparison in the Sasakian setting and we refer to \cite[Proposition 5.5]{agralee} and \cite[Theorem 1.1]{leeli} (see also \cite{bau,asm}) for details.

\begin{theorem}\label{t-comp} Let $(M^{2n+1}, \theta, J, g)$ be a $(2n+1)$-dimensional complete Sasakian manifold.
\begin{enumerate}

\item If $n=1$ and the Tanaka-Webster curvature satisfies
$$
\ricc_H(X,X)\geq 0 \quad \text{for all horizontal vector field }\,X,
$$
then, outside $\mathrm{Cut}(x_0)$ and in the sense of distributions,
$$
\Delta_b r \leq \frac{4}{r}
$$
and
$$
\vol B_R \leq C R^4\quad\forall R>1,
$$
for some $C>0$.

\item If $n\geq 2$ and the Tanaka-Webster curvature satisfies
$$
\ricc_H(X,X)\geq  \mathrm{R}(X,JX, X, JX)\geq 0 \quad \text{for all horizontal vector field }\,X,
$$
then, outside $\mathrm{Cut}(x_0)$ and in the sense of distributions,
$$
\Delta_b r \leq \frac{n+2}{r}
$$
and
$$
\vol B_R \leq C R^{2n+2}\quad\forall R>1,
$$
for some $C>0$.
\end{enumerate}
\end{theorem}

\

As a consequence of the Laplacian comparisons in Theorem \ref{t-comp} we obtain the following
\begin{corollary}\label{c-lowbou}
 Let $(M^{2n+1}, \theta, J, g)$ be a $(2n+1)$-dimensional complete Sasakian manifold satisfying the curvature assumptions of Theorem \ref{t-comp}. Let $u$ be a positive superharmonic function in $M$, i.e $u\in C^2(M)$ and
$$\Delta_{b}u\leq 0\quad\text{in }M.$$
Then, there exists a constant $C>0$ such that
$$
u(x) \geq \frac{C}{r(x)^{\max\{3,n+1\}}},
$$
for any $x\in M$ with $r(x)>1$.
\end{corollary}
\begin{proof}
Let 
$$
C=\min_{\partial B_1} u 
$$
and consider the function 
$$
v(x)=u(x)-\frac{C}{r(x)^q} \, \quad \text{ for $x\in B_1^c$}\, ,
$$
with $q=\max\{3,n+1\}$. Then, $v\geq 0$ on $\partial B_1$ and a direct computation shows that 
$$
\Delta_b v\leq 0 \quad \text{ in $B_1^c$}\, ,
$$
where one has to use the Laplacian comparison in Theorem \ref{t-comp}. Since 
$$
\liminf_{r(x)\to\infty} v (x) \geq 0 \, ,
$$
if $\inf_{B_1^c} v<0$, then $v$ attains its negative absolute minimum at a point in $\overline{B}_1^c$. By the strong maximum principle (see e.g. \cite{Bony}), then $v$ must be constant and negative on its domain, a contradiction. Thus $v\geq 0$ in $B^c_1$ and the claim follows.
\end{proof}

The next lemma shows that a bound on the $L^{\frac{2n+2}{n}}$-norm is equivalent to a bound on the total energy for a positive solution to the equation.
 
\begin{lemma} Let $(M^{2n+1}, \theta, J, g)$ be a $(2n+1)$-dimensional complete Sasakian manifold such that 
$$
\vol B_R \leq C R^{2n+2}\quad\forall R>1,
$$
for some $C>0$, and let $u$ be a positive solution of \eqref{eqcritn}. Then, if $\sigma\geq 0$,
$$
\int_{B_R\setminus B_{R/2}} u^{\frac{2n+2}{n}} =O(R^\sigma) \quad \Longleftrightarrow \quad \int_{B_R\setminus B_{R/2}} |\partial u|^2 =O(R^\sigma) \, .
$$

\end{lemma}

\begin{proof}	
Let $\varphi$ be a smooth real cut-off function such that $\varphi\equiv 1$ in $B_R\setminus B_{R/2}$, $\varphi\equiv 0$ in $B_{R/4}\cup B_{2R}^c$ and such that $\vert \partial\varphi\vert\leq \frac{c}{R}$. Multiplying \eqref{eqcritn} by $u\varphi^s$, where $s\geq 2n+2$, and integrating by parts we obtain
$$
2\int_{B_{2R}} |\partial u|^2\varphi^s + s\int_{B_{2R}} u\varphi^{s-1}\left(u_\alpha\varphi_{\bar{\alpha}}+ u_{\bar{\alpha}}\varphi_{\alpha}\right) = 2n^2 \int_{B_{2R}} u^{\frac{2+2n}{n}}\varphi^s\, .
$$
Moreover, from Young's inequality (twice), for every $\varepsilon>0$
\begin{align*}
\left|\int_{B_{2R}} u\varphi^{s-1}\left(u_\alpha\varphi_{\bar{\alpha}}+ u_{\bar{\alpha}}\varphi_{\alpha}\right)\right| & \leq \varepsilon\int_{B_{2R}} |\partial u|^2\varphi^s + \frac{1}{4\varepsilon} \int_{B_{2R}} u^2\varphi^{s-2}|\partial\varphi|^2 \\
&\leq \varepsilon\int_{B_{2R}} |\partial u|^2\varphi^s +\varepsilon\int_{B_{2R}} u^{\frac{2+2n}{n}}\varphi^s + C\int_{B_{2R}} \varphi^{s-2n-2}|\partial\varphi|^{2n+2}\\ 
&\leq \varepsilon\int_{B_{2R}} |\partial u|^2\varphi^s +\varepsilon\int_{B_{2R}} u^{\frac{2+2n}{n}}\varphi^s + C R^{-2-2n}\mathrm{Vol}B_{2R} \\
&\leq \varepsilon\int_{B_{2R}} |\partial u|^2\varphi^s +\varepsilon\int_{B_{2R}} u^{\frac{2+2n}{n}}\varphi^s + C\, ,
\end{align*}
for some $C>0$. Now, assuming 
$$
\int_{B_R\setminus B_{R/2}} u^{\frac{2n+2}{n}} \leq C R^\sigma  \, \quad \forall R>1\, , 
$$
holds, by choosing $\varepsilon>0$ small enough we have 
\begin{align*}
\int_{B_R\setminus B_{R/2}} |\partial u|^2 &\leq  C\left( \int_{B_{2R}\setminus B_{R/4}} u^{\frac{2+2n}{n}}+1\right)\\
&= C\left(\int_{B_{2R}\setminus B_{R}} u^{\frac{2+2n}{n}} + \int_{B_{2}\setminus B_{R/2}} u^{\frac{2+2n}{n}} + \int_{B_{R/2}\setminus B_{R/4}} u^{\frac{2+2n}{n}} +1\right) \leq CR^{\sigma}\, ,
\end{align*}
for all $R>2$. By continuity we deduce 
$$
\int_{B_R\setminus B_{R/2}} |\partial u|^2 \leq C R^{\sigma} \quad \forall R\geq 1\, ,
$$
for some $C>0$. The reverse implication can be proved in a similar way.
\end{proof}

\begin{corollary}
Let $(M^{2n+1}, \theta, J, g)$ be a $(2n+1)$-dimensional complete Sasakian manifold such that 
$$
\vol B_R \leq C R^{2n+2}\quad\forall R>1,
$$
for some $C>0$, and let $u$ be a positive solution of \eqref{eqcritn}. Then 
$$
u\in L^{\frac{2(n+2)}{n}}(M) \quad \Longleftrightarrow \quad |\partial u|\in L^2(M)\, . 
$$
\end{corollary}

Finally, in the next lemma we show that the decay assumption at infinity on a positive finite energy solution $u$ is satisfied, provided the manifold supports the Folland-Stein-Sobolev inequality. 

\begin{lemma}\label{Moser}
Let $(M^{2n+1}, \theta, J, g)$ be a $(2n+1)$-dimensional complete strictly pseudo-convex orientable CR manifold of hypersurface type supporting the following Folland-Stein-Sobolev inequality 
$$
\left(\int_M f^{\frac{2n+2}{n}}\right)^{\frac{2n}{2n+2}}\leq C \int_M |\partial f|^2 \quad \text{ for all } f\in\mathcal{D}^{1,2}(M)\, .
$$
Then, any positive solution $u\in L^{\frac{2n+2}{2}}(M)$ of \eqref{eqcritn} tends to zero at infinity.
\end{lemma}

\begin{proof}
The result can be proved by using a standard argument as in the elliptic setting in $\mathbb{R}^m$, by testing the equation against suitable compactly supported test functions depending on arbitrarily large powers of the solution $u$, and using the resulting integral estimates to perform a Moser iteration argument, which only requires the validity of the Folland-Stein-Sobolev inequality. We omit the details and we refer e.g. to \cite{Serrin}, \cite[Appendix E]{Peral}, \cite[Lemma 2.1]{Vet} and \cite[Proposition 2.2]{MurSoa}.
\end{proof}

\

\section{Integral estimates}\label{int_est}

In all this section, $(M^{2n+1}, \theta, J, g)$ will be a $(2n+1)$-dimensional complete Sasakian manifold, $f$, defined in \eqref{def_f}, satisfies \eqref{eq_f} and $g$, defined in \eqref{eq_g}, satisfies \eqref{eq_g}.

Moreover, given $R>0$, $\eta$ will denote a smooth real cut-off function such that $\eta\equiv 1$ in $B_{R/2}$, $\eta\equiv 0$ in $B_R^c$ and such that $\vert \partial\eta\vert\leq \frac{c}{R}$ in $A_R:=B_R\setminus B_{R/2}$.

For the reader's convenience we collect here some notations that we are going to use in the proofs of several lemmas appearing below. We define  
\begin{align*}
\mathcal{I}_1=& \int_M \mathcal{Q}\Psi^{-\beta}\eta^s \, , \\ 
\mathcal{I}_2=&\int_{A_R}e^{2(n-1)f} \vert g\vert^2\Psi^{-\beta} \eta^{s-2}\, , \\
\mathcal{Y}_0=&\int_{A_R}e^{2nf}\Psi^{-\beta}|\partial f|^2\eta^{s-2} \, ,   \\
\mathcal{Y}_1=& \int_{A_R} e^{2(n-1)f} \Psi^{-\beta} |\partial f|^2 \eta^{s-4}   \\
\hat{\mathcal{Y}}_1=&\int_{A_R} e^{2(n-1-\beta)f}  |\partial f|^2 \eta^{s-4} \\
\mathcal{Y}_3=&\int_{A_R}e^{2(n-1)f} \Psi^{-\beta} |\partial f|^4 \eta^{s-2} \, ,
\end{align*}
where $\beta\geq 0$ is small enough, $s>0$ is large enough, $n\geq 1$ and where $f,g, \mathcal{Q},\Psi$ are given by \eqref{def_f}, \eqref{g}, \eqref{defQ} and \eqref{defPsi}, respectively.

\begin{lemma}\label{lemma_est}
Let $n\geq 1$, $s\geq 4$ and $0\leq \beta <1/12$, then there exists a constant $C>0$ such that for every $R>0$ the following  
\begin{align*}
\int_{M} \mathcal{Q}\Psi^{-\beta}\eta^s\leq C \left( \frac{1}{R^4} \int_{A_R} e^{2(n-1)f} \Psi^{-\beta} |\partial f|^2 \eta^{s-4} + \frac{1}{R^2}\int_{A_R} e^{2(n-1)f} \Psi^{-\beta} |\partial f|^4 \eta^{s-2} \right)\,,
\end{align*}
holds, where $\mathcal{Q}$ and $\Psi$ are given by \eqref{defQ} and \eqref{defPsi} and $\eta$ is a cut-off function as above.

\end{lemma}

\begin{proof}

Given $R>0$ and $\eta$ as above, from Lemma \ref{lemma0}, \eqref{eq_g} and integrating by parts we have
\begin{align}\label{eq13}\nonumber
\mathcal{I}_1=\int_{M} \mathcal{Q}\Psi^{-\beta}\eta^s\leq & C \int_{M}e^{2(n-1)f}\vert g\vert^2\Psi^{-\beta} \vert\partial\eta\vert^2 \eta^{s-2} \\ 
\leq& \, \dfrac{C}{R^2} \int_{A_R}e^{2(n-1)f} \vert g\vert^2\Psi^{-\beta} \eta^{s-2}=:\dfrac{C}{R^2}\mathcal{I}_2  \\ \nonumber
=& -\dfrac{C}{R^2} \int_{A_R} e^{2(n-1)f} \mathrm{Re} (f_{\alpha\bar{\alpha}}\bar{g}) \Psi^{-\beta} \eta^{s-2}\\ \nonumber
=&:\dfrac{C}{R^2}\left[\mathcal{J}_1+(s-2)\mathcal{J}_2+2(n-1)\mathcal{J}_3+\beta\mathcal{J}_4 \right]\, ,
\end{align}
where 
\begin{align*}
\mathcal{J}_1:=& \int_{A_R} e^{2(n-1)f} \mathrm{Re} (f_{\alpha}\bar{g}_{\bar{\alpha}}) \Psi^{-\beta} \eta^{s-2} \\
\mathcal{J}_2:=& \int_{A_R} e^{2(n-1)f} \mathrm{Re} (f_{\alpha}\eta_{\bar{\alpha}}\bar{g}) \Psi^{-\beta} \eta^{s-3} \\
\mathcal{J}_3:=& \int_{A_R} e^{2(n-1)f} \mathrm{Re} (\bar{g}) \Psi^{-\beta} |\partial f|^2 \eta^{s-2} \\
\mathcal{J}_4:=& -\int_{A_R} e^{2(n-1)f} \mathrm{Re} (f_\alpha\bar{g}\Psi_{\bar{\alpha}}) \Psi^{-\beta-1} \eta^{s-2}\, . 
\end{align*}
We now estimate each of the previous integrals by using Young's inequality and \eqref{g}. We start with $\mathcal{J}_1$, using \eqref{gbarra} and the definition of $\mathcal{Q}$. Since 
\begin{equation}\label{stimaEFG}
\vert D_{\bar\alpha}+E_{\bar\alpha}-G_{\bar\alpha}\vert =\vert  (D_{\bar\alpha}+G_{\bar\alpha})+(E_{\bar\alpha}-G_{\bar\alpha})-G_{\bar\alpha}\vert\leq \sqrt{3}\sqrt\mathcal{Q} e^{-(n-1)f},
\end{equation}
then
\begin{align*}
\mathrm{Re}(f_\alpha \bar{g}_{\bar{\alpha}})= &\mathrm{Re} \left( f_{\alpha} \left( D_{\bar\alpha}+E_{\bar\alpha}-G_{\bar\alpha} \right)+ 2 \bar{g} |\partial f|^2 \right)  \\
\leq & \sqrt{3} |\partial f| \sqrt{\mathcal{Q}} e^{-(n-1)f} + 2|\partial f|^4 + 2 |\partial f|^2 e^{2f}\\
\leq & \sqrt{3} |\partial f| \sqrt{\mathcal{Q}} e^{-(n-1)f} + 4|\partial f|^4 + \dfrac{1}{2}  e^{4f}\, .
\end{align*}
Hence, 
\begin{align*}
\mathcal{J}_1\leq & \sqrt{3} \int_{A_R} e^{(n-1)f} \sqrt{\mathcal{Q}}|\partial f| \Psi^{-\beta}\eta^{s-2}+ \frac{1}{2} \int_{A_R} e^{2(n+1)f} \Psi^{-\beta}  \eta^{s-2} + 4 \int_{A_R} e^{2(n-1)f} \Psi^{-\beta} |\partial f|^4 \eta^{s-2} \\
\leq & \frac{\varepsilon R^2}{2} \int_{A_R} \mathcal{Q}\Psi^{-\beta}\eta^s + \frac{C}{R^2}\int_{A_R} e^{2(n-1)f} \Psi^{-\beta} |\partial f|^2 \eta^{s-4} + \frac{1}{2} \int_{A_R} e^{2(n+1)f} \Psi^{-\beta}  \eta^{s-2}\\
& + 4 \int_{A_R} e^{2(n-1)f} \Psi^{-\beta} |\partial f|^4 \eta^{s-2} \\
=:& \frac{\varepsilon R^2}{2}\mathcal{I}_1 + \frac{C}{R^2}\mathcal{Y}_1 +\frac{1}{2}\mathcal{Y}_2 +4\mathcal{Y}_3\, , 
\end{align*}
for every $\varepsilon>0$ and some $C=C(\varepsilon)>0$. Similarly we obtain 
\begin{align*}
(s-2)\mathcal{J}_2\leq& \varepsilon\int_{A_R}e^{2(n-1)f} \vert g\vert^2\Psi^{-\beta} \eta^{s-2}+ \frac{C}{R^2} \int_{A_R} e^{2(n-1)f}  \Psi^{-\beta}|\partial f|^2 \eta^{s-4} \\
=& \varepsilon\mathcal{I}_2 + \frac{C}{R^2}\mathcal{Y}_1\, , 
\end{align*}
for every $\varepsilon>0$ and some $C=C(\varepsilon)>0$. Analogously we have 
\begin{align*}
2(n-1)\mathcal{J}_3 =& 2(n-1) \int_{A_R} e^{2(n-1)f} \Psi^{-\beta} |\partial f|^4 \eta^{s-2} +2(n-1) \int_{A_R} e^{2nf} \Psi^{-\beta} |\partial f|^2 \eta^{s-2} \\
\leq & \varepsilon \int_{A_R} e^{2(n+1)f} \Psi^{-\beta}  \eta^{s-2}+ C  \int_{A_R} e^{2(n-1)f} \Psi^{-\beta} |\partial f|^4 \eta^{s-2} \\
=& \varepsilon\mathcal{Y}_2 + C\mathcal{Y}_3  \, ,
\end{align*}
for every $\varepsilon>0$ and some $C=C(\varepsilon)>0$. Moreover we get 
\begin{align*}
\mathcal{J}_4 \leq & 2\int_{A_R} e^{(n-1)f} |g|^2 \sqrt{\mathcal{Q}} |\partial f| \Psi^{-\beta-1} e^{-2f}\eta^{s-2}\\
\leq & \frac{\varepsilon R^2}{2} \int_{M} \mathcal{Q}\Psi^{-\beta}\eta^s + \frac{C}{R^2} \int_{A_R} e^{2(n-1)f}  \Psi^{-\beta}|\partial f|^2 \eta^{s-4}\\
=&\frac{\varepsilon R^2}{2} \mathcal{I}_1+\frac{C}{R^2}\mathcal{Y}_1 \, , 
\end{align*}
for every $\varepsilon>0$ and some $C=C(\varepsilon)>0$, where we used \eqref{p2}. Summing up, 
\begin{align*}
\mathcal{I}_2 \leq \varepsilon R^2\mathcal{I}_1 + \frac{C}{R^2}\mathcal{Y}_1 +\left(\frac{1}{2}+\varepsilon\right)\mathcal{Y}_2 +C\mathcal{Y}_3 + \varepsilon \mathcal{I}_2 \, .
\end{align*}
We now observe that
\begin{align*}
\mathcal{I}_2 = \frac{1}{3}\mathcal{I}_2  + \frac{2}{3}\mathcal{I}_2 \geq \frac{1}{3}\mathcal{I}_2  + \frac{2}{3}\mathcal{Y}_2\, , 
\end{align*}
hence, for every $\varepsilon>0$ sufficiently small we obtain 
\begin{align}\label{dario}
\mathcal{I}_2 \leq \varepsilon C R^2\mathcal{I}_1 + \frac{C}{R^2}\mathcal{Y}_1  +C\mathcal{Y}_3  \, .
\end{align}
Combining \eqref{eq13} and \eqref{dario} we get 
\begin{align*}
\mathcal{I}_1 \leq \dfrac{C}{R^2} \mathcal{I}_2 \leq  \dfrac{C}{R^2}  \left( \varepsilon  R^2\mathcal{I}_1 + \frac{1}{R^2}\mathcal{Y}_1  +\mathcal{Y}_3 \right)
\end{align*}
and the claim follows by choosing $\varepsilon$ sufficiently small.
\end{proof}

\begin{lemma}\label{l-est2}
Let $n=1,2$, $s\geq 4+2n$ and $0< \beta <n$, then there exists a constant $C>0$ such that for every $R>0$ the following holds
\begin{align*}
\int_{A_R} e^{2(n-1)f} \Psi^{-\beta} |\partial f|^2 \eta^{s-4}\leq \frac{C}{R^{2(n-\beta)}} \vol B_R \,,
\end{align*}
where $\Psi$ is given by \eqref{defPsi} and $\eta$ is a cut-off function as above.
\end{lemma}
\begin{proof}
Firstly we observe that, being $\beta>0$, by \eqref{defPsi} we have 
\begin{equation}\label{trick}
\Psi^{-\beta}= \vert g\vert^{-2\beta} e^{2\beta f} \leq e^{2\beta f}\left( \vert\partial f\vert^4 + e^{4f}\right)^{-\beta} \leq e^{-2\beta f}
\end{equation}
Then, 
\begin{align}\label{fabio}
\mathcal{Y}_1 = \int_{A_R} e^{2(n-1)f} \Psi^{-\beta} |\partial f|^2 \eta^{s-4} \leq \int_{A_R} e^{2(n-1-\beta)f}  |\partial f|^2 \eta^{s-4}=: \hat{\mathcal{Y}}_1\,. 
\end{align}
We now have from \eqref{g}
\begin{align}\label{clelia}
\hat{\mathcal{Y}}_1 = & \int_{A_R} e^{2(n-1-\beta)f}  \mathrm{Re}(g) \eta^{s-4} - \int_{A_R} e^{2(n-\beta)f}  \eta^{s-4} \\ \nonumber 
= & -\frac{1}{n} \int_{A_R} e^{2(n-1-\beta)f}  \mathrm{Re}(f_{\alpha\bar{\alpha}}) \eta^{s-4} - \int_{A_R} e^{2(n-\beta)f}  \eta^{s-4} \\ \nonumber 
= & \frac{2(n-1-\beta)}{n} \hat{\mathcal{Y}}_1 + \frac{s-4}{n} \int_{A_R} e^{2(n-1-\beta)f}  \mathrm{Re}(f_{\alpha} \eta_{\bar{\alpha}}) \eta^{s-5} - \int_{A_R} e^{2(n-\beta)f}  \eta^{s-4}
\end{align}
thus, being $2(1+\beta)-n>0$ by our assumptions, 
\begin{align*}
\hat{\mathcal{Y}}_1 \leq  & C \left( \dfrac{s-4}{n} \int_{A_R} e^{2(n-1-\beta)f} |\partial f| |\partial\eta| \eta^{s-5} - \int_{A_R} e^{2(n-\beta)f}  \eta^{s-4}\right) \\
\leq & C\left(\varepsilon\hat{\mathcal{Y}}_1 +  \hat{C} \int_{A_R} e^{2(n-1-\beta)f}  \eta^{s-6} |\partial\eta|^2 - \int_{A_R} e^{2(n-\beta)f}  \eta^{s-4} \right) \, ,
\end{align*}
for every $\varepsilon>0$ and some $C,\hat{C}>0$. By choosing $\varepsilon$ sufficiently small we deduce 
\begin{align*}
\hat{\mathcal{Y}}_1 \leq  &C\left( \hat{C} \int_{A_R} e^{2(n-1-\beta)f}  \eta^{s-6} |\partial\eta|^2 - \int_{A_R} e^{2(n-\beta)f}  \eta^{s-4} \right) \\
\leq & C\left( \hat{C}\delta \int_{A_R} e^{2(n-\beta)f}  \eta^{s-4} + \tilde{C}   \int_{A_R}  \eta^{s-4-2n+2\beta} |\partial\eta|^{2(n-\beta)} - \int_{A_R} e^{2(n-\beta)f}  \eta^{s-4} \right)
\end{align*}
for every $\delta>0$ and some $C,\hat{C},\tilde{C}>0$. By choosing $\delta$ sufficiently small we deduce, from \eqref{fabio},
\begin{align}\label{eq-y1h}
\mathcal{Y}_1 \leq \hat{\mathcal{Y}}_1 \leq C\int_{A_R}  \eta^{s-4-2n+2\beta} |\partial\eta|^{2(n-\beta)}\leq  \frac{C}{R^{2(n-\beta)}} \vol B_R
\end{align}
which is the thesis.
\end{proof}

\begin{lemma}\label{lemma33}
Let $(M^{3}, \theta, J, g)$ be a $3$-dimensional ($n=1$) complete Sasakian manifold with nonnegative Tanaka-Webster scalar curvature. Let $s> 8$, $0< \beta <1/27$ and $\eps>0$, then there exists a constant $C>0$ such that for every $R>0$ the following holds
\begin{align*}
\int_{A_R}  \Psi^{-\beta} |\partial f|^4 \eta^{s-2}\leq C \left(\eps R^2 \int_{A_R}\Psi^{-\beta}\mathcal{Q}\eta^s+\frac{1}{R^2}\vol B_R\right)\, , 
\end{align*}
where $\Psi$ is given by \eqref{defPsi} and $\eta$ is a cut-off function as above.
\end{lemma}

\begin{proof}
By \eqref{defPsi} we have
\begin{align*}
\mathcal{Y}_3=&\int_{A_R}  \Psi^{-\beta} |\partial f|^4 \eta^{s-2}\leq \int_{A_R} e^{-2\beta f}  |\partial f|^4 \eta^{s-2} =:\hat{\mathcal{Y}}_3\\
=&- \int_{A_R} e^{-2\beta f} \mathrm{Re}(f_{\alpha\bar{\alpha}}) |\partial f|^2 \eta^{s-2} - \int_{A_R} e^{2(1-\beta)f}  |\partial f|^2 \eta^{s-2} \\ 
=&  -2\beta \hat{\mathcal{Y}}_3 + (s-2) \int_{A_R} e^{-2\beta f}  \mathrm{Re}(f_\alpha\eta_{\bar{\alpha}}) |\partial f|^2 \eta^{s-3} \\
&+  \int_{A_R} e^{-2\beta f}  \mathrm{Re}\left( f_\alpha  |\partial f|^2_{\bar{\alpha}}\right) \eta^{s-2}  -\int_{A_R} e^{2(1-\beta)f}  |\partial f|^2 \eta^{s-2} \, .
\end{align*}
Now we note that 
\begin{align*}
 \int_{A_R} e^{-2\beta f}  \mathrm{Re}\left( f_\alpha  |\partial f|^2_{\bar{\alpha}}\right) \eta^{s-2} = &  \int_{A_R} e^{-2\beta f}  \mathrm{Re}\left( f_\alpha (D_{\bar{\alpha}}+E_{\bar{\alpha}}) \right) \eta^{s-2} \\
& +  \int_{A_R} e^{-2\beta f} |\partial f|^2\mathrm{Re}\left( \bar{g}  \right) \eta^{s-2}\\
& - 2 \int_{A_R} e^{2(1-\beta)f} |\partial f|^2 \eta^{s-2} \\
=&  \int_{A_R} e^{-2\beta f} \mathrm{Re}\left( f_\alpha (D_{\bar{\alpha}}+E_{\bar{\alpha}}) \right) \eta^{s-2} \\
& +\hat{\mathcal{Y}}_3 - \int_{A_R} e^{2(1-\beta)f}  |\partial f|^2 \eta^{s-2}
\end{align*}
where we used \eqref{tensg}. Using this identity in the previous one we get
\begin{align}\label{eq-bobo}
2\beta \hat{\mathcal{Y}}_3=& (s-2) \int_{A_R} e^{-2\beta f}  \mathrm{Re}(f_\alpha\eta_{\bar{\alpha}}) |\partial f|^2 \eta^{s-3} \\\nonumber
& + \int_{A_R} e^{-2\beta f} \mathrm{Re}\left( f_\alpha (D_{\bar{\alpha}}+E_{\bar{\alpha}}) \right) \eta^{s-2}-2\int_{A_R} e^{2(1-\beta)f}  |\partial f|^2 \eta^{s-2} \, .
\end{align}
Since $n=1$, we follow the proof in \cite[Section 3]{CLMR}; we include the details in order to make the presentation self contained. We have
\begin{align*}
2\beta\hat{\mathcal{Y}}_3=&  (s-2)\int_{A_R} e^{-2\beta f}  \mathrm{Re}(f_\alpha\eta_{\bar{\alpha}}) |\partial f|^2 \eta^{s-3} \\
& + \int_{A_R} e^{-2\beta f} \mathrm{Re}\left( f_\alpha (D_{\bar{\alpha}}+E_{\bar{\alpha}}) \right) \eta^{s-2}-2\int_{A_R} e^{2(1-\beta)f}  |\partial f|^2 \eta^{s-2} \, .
\end{align*}
Using Young and the following inequality (which actually holds for every $n\geq 1$)
\begin{align}\label{eq-estest}
|D_{\bar{\alpha}}+ E_{\bar{\alpha}}|\leq \sqrt{2} \sqrt{\mathcal{Q}} e^{-(n-1)f}
\end{align}
we obtain
\begin{align*}
(2\beta-\eps)\hat{\mathcal{Y}}_3\leq&  \frac{C}{R^2}\int_{A_R} e^{-2\beta f}  |\partial f|^2 \eta^{s-4} +\sqrt{2} \int_{A_R}e^{-2\beta f}\sqrt{\mathcal{Q}}|\partial f|\eta^{s-2}-2\int_{A_R} e^{2(1-\beta)f}  |\partial f|^2 \eta^{s-2} \\
= &  \frac{C}{R^2}\int_{A_R} e^{-2\beta f}  |\partial f|^2 \eta^{s-4}+\sqrt{2} \int_{A_R}e^{-4\beta f}\Psi^{-\beta}|g|^{2\beta}\sqrt{\mathcal{Q}}|\partial f|\eta^{s-2}\\
&-2\int_{A_R} e^{2(1-\beta)f}  |\partial f|^2 \eta^{s-2} \\
\leq &  \frac{C}{R^2}\int_{A_R} e^{-2\beta f}  |\partial f|^2 \eta^{s-4}+\eps R^2 \int_{A_R}\Psi^{-\beta}\mathcal{Q}\eta^s+\frac{C}{R^2}\int_{A_R}e^{-8\beta f}\Psi^{-\beta}|g|^{4\beta}|\partial f|^2 \eta^{s-4}\,.
\end{align*}
We now use the following Young's inequality
$$
e^{-8\beta f} g\vert^{4\beta} \vert\partial f\vert^2 \vert  \eta^{s-4} \leq 2\beta \vert g\vert^2 \eta^{s-2} + (1-2\beta) e^{-\frac{8\beta f}{1-2\beta}}\vert\partial f\vert^{\frac{2}{1-2\beta}}\eta^{s-\frac{4-4\beta}{1-2\beta}}
$$
to get
\begin{align*}
(2\beta-\eps)\hat{\mathcal{Y}}_3 
\leq &  \frac{C}{R^2}\int_{A_R} e^{-2\beta f}  |\partial f|^2 \eta^{s-4}+\eps R^2 \int_{A_R}\Psi^{-\beta}\mathcal{Q}\eta^s\\
&+\frac{C}{R^2}\int_{A_R}e^{-\frac{8\beta}{1-2\beta} f}\Psi^{-\beta}|\partial f|^{\frac{2}{1-2\beta}}\eta^{s-\frac{4-4\beta}{1-2\beta}}+\frac{C}{R^2}\int_{A_R}\Psi^{-\beta}|g|^2\eta^{s-2}\,.
\end{align*}
Using the following Young's inequality
$$
e^{-\frac{8\beta f}{1-2\beta}}\Psi^{-\beta}\vert\partial f\vert^{\frac{2}{1-2\beta}}\eta^{s-\frac{4-4\beta}{1-2\beta}}\leq \frac{1}{2(1-2\beta)}\vert\partial f\vert^4e^{-16\beta f}\Psi^{-2\beta}\eta^{s-2} + \frac{1-4\beta}{2-4\beta}\eta^{s-\frac{6-8\beta}{1-4\beta}}\,
$$
together with \eqref{trick} and 
\begin{equation}\label{stima_f}
e^{-f}=u^{-1}\leq C R^3\quad\text{in }A_R\,,
\end{equation}
which follows immediately from Corollary \ref{c-lowbou}, we obtain
$$
e^{-\frac{8\beta f}{1-2\beta}}\Psi^{-\beta}\vert\partial f\vert^{\frac{2}{1-2\beta}}\eta^{s-\frac{4-4\beta}{1-2\beta}}\leq C \left(R^{54\beta}\vert\partial f\vert^4\Psi^{-\beta}\eta^{s-2} + 1\right)\,.
$$
Hence, from \eqref{eq-y1h} and \eqref{dario}, we get
\begin{align*}
(2\beta-\eps)\mathcal{Y}_3\leq & (2\beta-\eps)\hat{\mathcal{Y}}_3 \\
\leq &  \frac{C}{R^2}\int_{A_R} e^{-2\beta f}  |\partial f|^2 \eta^{s-4}+\eps R^2 \int_{A_R}\Psi^{-\beta}\mathcal{Q}\eta^s\\
&+\frac{C}{R^{2-54\beta}}\int_{A_R}\Psi^{-\beta}|\partial f|^{4}\eta^{s-2}+\frac{C}{R^2}\vol B_R+\frac{C}{R^{2}}\int_{A_R}\Psi^{-\beta}|g|^2\eta^{s-2}\\
= &  \frac{C}{R^2}\hat{\mathcal{Y}}_1+\eps R^2\mathcal{I}_1+\frac{C}{R^{2-54\beta}}\mathcal{Y}_3+\frac{C}{R^2}\vol B_R+\frac{C}{R^{2}}\mathcal{I}_2\\
\leq & \eps R^2 \mathcal{I}_1+\frac{C}{R^{2-54\beta}}\mathcal{Y}_3+\frac{C}{R^2}\vol B_R+\frac{C}{R^2}\left(\eps R^2 \mathcal{I}_1+\frac{C}{R^2}\mathcal{Y}_1+C\mathcal{Y}_3\right)\\
\leq &\eps R^2 \mathcal{I}_1+\frac{C}{R^{2-54\beta}}\mathcal{Y}_3+\frac{C}{R^2}\vol B_R\,,
\end{align*}
where $\mathcal{I}_2$ is defined in \eqref{eq13} and where we used Lemma \ref{l-est2}. Since $0<\beta<1/27$, choosing $\eps$ small enough and $R$ large enough we obtain the thesis.
\end{proof}

\begin{lemma}\label{l-est22}
Let $n=2$, $s> 4$, $0< \beta <1$ and $\eps>0$, then there exists a constant $C>0$ such that for every $R>0$ the following hold
\begin{align*}
\int_{A_R}e^{4f}\Psi^{-\beta}|\partial f|^2\eta^{s-2}\leq C\left( \int_{A_R}e^{(6-2\beta)f}\eta^{s-2} + \frac{1}{R^2}\int_{A_R}e^{(4-2\beta)f}\eta^{s-4} \right)
\end{align*}
and
\begin{align*}
\int_{A_R}  e^{2f}\Psi^{-\beta} |\partial f|^4 \eta^{s-2}\leq  C\left( \int_{A_R}e^{4f}\Psi^{-\beta}|\partial f|^2\eta^{s-2}+ \frac{1}{R^2}\int_{A_R} e^{2f}\Psi^{-\beta}|\partial f|^2 \eta^{s-4}+\varepsilon R^2\int_{M} \Psi^{-\beta}\mathcal{Q}\eta^s  \right)\, , 
\end{align*}
where $\Psi$ is given by \eqref{defPsi} and $\eta$ is a cut-off function as above.
\end{lemma}
\begin{proof} We have
\begin{align*}
\mathcal{Y}_0=&\int_{A_R}e^{4f}\Psi^{-\beta}|\partial f|^2\eta^{s-2} \leq \int_{A_R}e^{(4-2\beta)f}|\partial f|^2\eta^{s-2}=:\hat{\mathcal{Y}}_0 \\
=& \int_{A_R} e^{(4-2\beta)f}  \mathrm{Re}(g) \eta^{s-2} - \int_{A_R} e^{(6-2\beta)f}  \eta^{s-2} \\
= & -\frac{1}{2} \int_{A_R} e^{(4-2\beta)f}  \mathrm{Re}(f_{\alpha\bar{\alpha}}) \eta^{s-s} - \int_{A_R} e^{(6-2\beta)f}  \eta^{s-2} \\
= & (2-\beta) \hat{\mathcal{Y}}_0 + \frac{s-2}{2} \int_{A_R} e^{(4-2\beta)f} \mathrm{Re}( f_{\alpha} \eta_{\bar{\alpha}}) \eta^{s-3} - \int_{A_R} e^{(6-2\beta)f}  \eta^{s-2}\, .
\end{align*}
Hence, 
\begin{align*}
(1-\beta)\hat{\mathcal{Y}}_0=& \int_{A_R} e^{(6-2\beta)f}  \eta^{s-2}-\frac{s-2}{2} \int_{A_R} e^{(4-2\beta)f}  \mathrm{Re}(f_{\alpha} \eta_{\bar{\alpha}}) \eta^{s-3} \\
\leq & \eps\hat{\mathcal{Y}}_0 + \frac{C}{R^2} \int_{A_R} e^{(4-2\beta)f}\eta^{s-4} + \int_{A_R} e^{(6-2\beta)f}  \eta^{s-2} \, ,
\end{align*}
for every $\varepsilon>0$ and some $C>0$. By choosing $\varepsilon$ small enough we have 
\begin{align*}
\mathcal{Y}_0 \leq \hat{\mathcal{Y}}_0\leq  C\left( \int_{A_R} e^{(6-2\beta)f}  \eta^{s-2} + \frac{1}{R^2}\int_{A_R} e^{(4-2\beta)f} \eta^{s-4}\, ,  \right)
\end{align*}
and the first estimate follows. Concerning the second one, we have
\begin{align*}
\mathcal{Y}_3=&\int_{A_R} e^{2f} \Psi^{-\beta} |\partial f|^4 \eta^{s-2}\\
=&-\frac{1}{2} \int_{A_R} e^{2f}\Psi^{-\beta} \mathrm{Re}(f_{\alpha\bar{\alpha}}) |\partial f|^2 \eta^{s-2} - \int_{A_R} e^{4f}\Psi^{-\beta}  |\partial f|^2 \eta^{s-2} \\ 
=&  \mathcal{Y}_3+\frac{s-2}{2}\int_{A_R}e^{2f}\Psi^{-\beta}|\partial f|^2 \mathrm{Re}(\eta_{\bar\alpha}f_{\alpha})\eta^{s-3}+\frac12 \int_{A_R}e^{2f}\Psi^{-\beta}\mathrm{Re}(f_{\alpha} |\partial f|^2_{\bar\alpha})\eta^{s-2}\\
&- \frac\beta 2 \int_{A_R}e^{2f}\Psi^{-\beta-1}|\partial f|^2\mathrm{Re}(f_{\alpha} \Psi_{\bar\alpha})\eta^{s-2} -\int_{A_R} e^{4f}\Psi^{-\beta}  |\partial f|^2 \eta^{s-2}\, ,
\end{align*}
i.e.
\begin{align*}
0=& \frac{s-2}{2}\int_{A_R}e^{2f}\Psi^{-\beta}|\partial f|^2 \mathrm{Re}(\eta_{\bar\alpha}f_{\alpha})\eta^{s-3}+\frac12 \int_{A_R}e^{2f}\Psi^{-\beta}\mathrm{Re}(f_{\alpha} |\partial f|^2_{\bar\alpha})\eta^{s-2}\\
&- \frac\beta 2 \int_{A_R}e^{2f}\Psi^{-\beta-1}|\partial f|^2\mathrm{Re}(f_{\alpha} \Psi_{\bar\alpha})\eta^{s-2} -\int_{A_R} e^{4f}\Psi^{-\beta}  |\partial f|^2 \eta^{s-2}\\
=& \frac{s-2}{2}\int_{A_R}e^{2f}\Psi^{-\beta}|\partial f|^2 \mathrm{Re}(\eta_{\bar\alpha}f_{\alpha})\eta^{s-3}+\frac12 \int_{A_R}e^{2f}\Psi^{-\beta}\mathrm{Re}(f_{\alpha}(D_{\bar\alpha}+E_{\bar\alpha}))\eta^{s-2}\\
&+\frac12 \int_{A_R}e^{2f}\Psi^{-\beta}\mathrm{Re}(\bar g)|\partial f|^2\eta^{s-2}\\
&- \frac\beta 2 \int_{A_R}e^{2f}\Psi^{-\beta-1}|\partial f|^2\mathrm{Re}(f_{\alpha} \Psi_{\bar\alpha})\eta^{s-2} -2\int_{A_R} e^{4f}\Psi^{-\beta}  |\partial f|^2 \eta^{s-2}\\
=&\frac{s-2}{2}\int_{A_R}e^{2f}\Psi^{-\beta}|\partial f|^2 \mathrm{Re}(\eta_{\bar\alpha}f_{\alpha})\eta^{s-3}+\frac12 \int_{A_R}e^{2f}\Psi^{-\beta}\mathrm{Re}(f_{\alpha}(D_{\bar\alpha}+E_{\bar\alpha}))\eta^{s-2}\\
&+\frac12\mathcal{Y}_3- \frac\beta 2 \int_{A_R}e^{2f}\Psi^{-\beta-1}|\partial f|^2\mathrm{Re}(f_{\alpha} \Psi_{\bar\alpha})\eta^{s-2} -\frac32\int_{A_R} e^{4f}\Psi^{-\beta}  |\partial f|^2 \eta^{s-2}\,.
\end{align*}
Therefore, using Young's inequality, \eqref{eq-estest} and \eqref{p2}, we obtain
\begin{align*}
\frac12\mathcal{Y}_3\leq &
\eps \mathcal{Y}_3+\frac{C}{R^2}\int_{A_R}e^{2f}\Psi^{-\beta}|\partial f|^2\eta^{s-4}+\eps R^2\int_{A_R}\Psi^{-\beta}\mathcal{Q}\eta^s \\
&+ \frac{C}{R^2} \int_{A_R}e^{-2f}\Psi^{-\beta-2}|\partial f|^6|g|^2\eta^{s-4} +\frac32\int_{A_R} e^{4f}\Psi^{-\beta}  |\partial f|^2 \eta^{s-2}\\
&\leq \eps \mathcal{Y}_3+\frac{C}{R^2}\int_{A_R}e^{2f}\Psi^{-\beta}|\partial f|^2\eta^{s-4}+\eps R^2\int_{A_R}\Psi^{-\beta}\mathcal{Q}\eta^s +\frac32\int_{A_R} e^{4f}\Psi^{-\beta}  |\partial f|^2 \eta^{s-2}\,,
\end{align*}
since $\Psi^{-2}|\partial f|^4 |g|^2 \leq e^{4f}$. Hence
\begin{align*}
\mathcal{Y}_3\leq &\frac{C}{R^2}\mathcal{Y}_1+\eps R^2\mathcal{I}_1+C\mathcal{Y}_0\,
\end{align*}
and the conclusion follows.
\end{proof}

\begin{lemma}\label{l-estn}
Let $n\geq 3$ and $s\geq 4+2n$ then, for every $\varepsilon>0$ small enough and for every $0\leq\beta<\frac{1}{12}$, there exists a constant $C>0$ such that for every $R>0$ the following 
\begin{align*}
\int_{A_R} e^{2(n-1)f}|\partial f|^2 \eta^{s-4}\leq C\left( \int_{A_R} e^{2nf} \eta^{s-4} + \frac{1}{R^{2n}}\vol B_R \right) \, ,
\end{align*}
\begin{align*}
\int_{A_R}e^{2nf}|\partial f|^2\eta^{s-2}\leq C\left( \int_{A_R}e^{2(n+1)f}\eta^{s-2} + \frac{1}{R^2}\int_{A_R}e^{2nf}\eta^{s-4} \right)
\end{align*}
and 
\begin{align*}
\int_{A_R} e^{2(n-1)f}|\partial f|^4 \eta^{s-2}\leq C\left( \int_{A_R}e^{2nf}|\partial f|^2\eta^{s-2}+ \frac{1}{R^2}\int_{A_R} e^{2(n-1)f}|\partial f|^2 \eta^{s-4}+\varepsilon R^2\int_{M} \mathcal{Q}\eta^s  \right)\,, 
\end{align*}
hold, where $\eta$ is a cut-off function as above. 

\end{lemma}

\begin{proof} We firstly observe that, as in \eqref{clelia} with $\beta=0$, 
\begin{align*}
\mathcal{Y}_1=\int_{A_R} e^{2(n-1)f}|\partial f|^2 \eta^{s-4}= & \int_{A_R} e^{2(n-1)f}  \mathrm{Re}(g) \eta^{s-4} - \int_{A_R} e^{2nf}  \eta^{s-4} \\
= & -\frac{1}{n} \int_{A_R} e^{2(n-1)f}  \mathrm{Re}(f_{\alpha\bar{\alpha}}) \eta^{s-4} - \int_{A_R} e^{2nf}  \eta^{s-4} \\
= & \frac{2(n-1)}{n} \mathcal{Y}_1 + \frac{s-4}{n} \int_{A_R} e^{2(n-1)f}  \mathrm{Re}(f_{\alpha} \eta_{\bar{\alpha}}) \eta^{s-5} - \int_{A_R} e^{2nf}  \eta^{s-4}\, . 
\end{align*}
Hence, 
\begin{align*}
\frac{n-2}{n} \mathcal{Y}_1 =&-\frac{s-4}{n} \int_{A_R} e^{2(n-1)f}  \mathrm{Re}(f_{\alpha} \eta_{\bar{\alpha}}) \eta^{s-5} + \int_{A_R} e^{2nf}  \eta^{s-4} \\
\leq & \frac{s-4}{n} \int_{A_R} e^{2(n-1)f}  |\partial f| |\partial\eta| \eta^{s-5} + \int_{A_R} e^{2nf}  \eta^{s-4} \\
\leq & \varepsilon \mathcal{Y}_1  +C \int_{A_R}e^{2(n-1)f}   |\partial\eta|^2 \eta^{s-6} + \int_{A_R} e^{2nf}  \eta^{s-4} \, , 
\end{align*}
for every $\varepsilon>0$ and some $C>0$. By choosing $\varepsilon$ small enough we have 
\begin{align*}
\mathcal{Y}_1 \leq & C\left( \int_{A_R}e^{2(n-1)f}   |\partial\eta|^2 \eta^{s-6} + \int_{A_R} e^{2nf}  \eta^{s-4}\right)  \\
\leq & C\left( \int_{A_R} e^{2nf} \eta^{s-4} + \int_{A_R} \eta^{s-4-2n} |\partial\eta|^{2n}\right)\, ,
\end{align*}
and the first estimate follows where we used Young's inequality. Similarly, 
\begin{align*}
\mathcal{Y}_0:=&\int_{A_R}e^{2nf}|\partial f|^2\eta^{s-2} = \int_{A_R} e^{2nf}  \mathrm{Re}(g) \eta^{s-2} - \int_{A_R} e^{2(n+1)f}  \eta^{s-2} \\
= & -\frac{1}{n} \int_{A_R} e^{2nf}  \mathrm{Re}(f_{\alpha\bar{\alpha}}) \eta^{s-s} - \int_{A_R} e^{2(n+1)f}  \eta^{s-2} \\
= & 2 \mathcal{Y}_0 + \frac{s-2}{n} \int_{A_R} e^{2nf}  \mathrm{Re}(f_{\alpha} \eta_{\bar{\alpha}}) \eta^{s-3} - \int_{A_R} e^{2(n+1)f}  \eta^{s-2}\, .
\end{align*}
Hence, 
\begin{align*}
\mathcal{Y}_0=& \int_{A_R} e^{2(n+1)f}  \eta^{s-2}-\frac{s-2}{n} \int_{A_R} e^{2nf}  \mathrm{Re}(f_{\alpha} \eta_{\bar{\alpha}}) \eta^{s-3} \\
\leq & \varepsilon\mathcal{Y}_0 + C \int_{A_R} e^{2nf} |\partial\eta|^2 \eta^{s-4} + \int_{A_R} e^{2(n+1)f}  \eta^{s-2} \, ,
\end{align*}
for every $\varepsilon>0$ and some $C>0$. By choosing $\varepsilon$ small enough we have 
\begin{align*}
\mathcal{Y}_0 \leq C\left( \int_{A_R} e^{2(n+1)f}  \eta^{s-2} + \int_{A_R} e^{2nf} |\partial\eta|^2 \eta^{s-4}\, ,  \right)
\end{align*}
and the second estimate follows. Finally, 
\begin{align*}
\mathcal{Y}_3 = &\int_{A_R} e^{2(n-1)f}|\partial f|^4 \eta^{s-2} = \int_{A_R} e^{2(n-1)f} \mathrm{Re}(g) |\partial f|^2\eta^{s-2} - \int_{A_R} e^{2nf}|\partial f|^2 \eta^{s-2}\\ 
=& -\frac{1}{n}\int_{A_R}e^{2(n-1)f} \mathrm{Re}(f_{\alpha\bar{\alpha}}) |\partial f|^2\eta^{s-2} - \int_{A_R} e^{2nf}|\partial f|^2 \eta^{s-2}\\
=& \frac{1}{n}\int_{A_R}e^{2(n-1)f} \mathrm{Re}(f_{\alpha}|\partial f|^2_{\bar{\alpha}}) \eta^{s-2} + \frac{2(n-1)}{n}\int_{A_R}e^{2(n-1)f} |\partial f|^4 \eta^{s-2} \\ 
&+ \frac{s-2}{2}\int_{A_R} e^{2(n-1)f} \mathrm{Re}(f_{\bar{\alpha}}\eta_{\bar{\alpha}})|\partial f|^2\eta^{s-3} - \int_{A_R} e^{2nf}|\partial f|^2 \eta^{s-2} \, ,
\end{align*}
hence 
\begin{align*}
\frac{n-2}{n} \mathcal{Y}_3 =&  \int_{A_R} e^{2nf}|\partial f|^2 \eta^{s-2}-\frac{1}{n}\int_{A_R}e^{2(n-1)f} \mathrm{Re}(f_{\alpha}|\partial f|^2_{\bar{\alpha}}) \eta^{s-2} \\ &- \frac{s-2}{2}\int_{A_R} e^{2(n-1)f} \mathrm{Re}(f_{\bar{\alpha}}\eta_{\bar{\alpha}})|\partial f|^2\eta^{s-3}\\
=& \int_{A_R} e^{2nf}|\partial f|^2 \eta^{s-2} - \frac{1}{n}\int_{A_R}e^{2(n-1)f} \mathrm{Re}\left(f_{\alpha}(D_{\bar{\alpha}}+E_{\bar{\alpha}})\right) \eta^{s-2}  \\ 
& - \frac{1}{n}\int_{A_R}e^{2(n-1)f} \mathrm{Re}(\bar{g})|\partial f|^2 \eta^{s-2}+  \frac{2}{n}\int_{A_R}e^{2nf} |\partial f|^2\eta^{s-2}  \\  
&- \frac{s-2}{2}\int_{A_R} e^{2(n-1)f} \mathrm{Re}(f_{\bar{\alpha}}\eta_{\bar{\alpha}})|\partial f|^2\eta^{s-3} \\
\leq & \left(1+\frac{1}{n}\right) \int_{A_R} e^{2nf}|\partial f|^2 \eta^{s-2} - \frac{1}{n} \mathcal{Y}_3 + \frac{\sqrt{2}}{n}\int_{A_R}e^{(n-1)f}|\partial f|\sqrt{\mathcal{Q}} \eta^{s-2} \\ 
& + \delta \mathcal{Y}_3 + C \int_{A_R} e^{2(n-1)f}|\partial f|^2 |\partial\eta|^2 \eta^{s-4}\\
\leq & \left(1+\frac{1}{n}\right) \int_{A_R} e^{2nf}|\partial f|^2 \eta^{s-2} - \frac{1}{n} \mathcal{Y}_3 + \varepsilon R^2 \int_{A_R}\mathcal{Q} \eta^{s} \\  
&+ \frac{C}{R^2}\int_{A_R} e^{2(n-1)f}|\partial f|^2 \eta^{s-4} + \delta \mathcal{Y}_3 \,, 
\end{align*}
for every $\delta,\varepsilon>0$ and some $C>0$, where we used \eqref{tensg} and \eqref{eq-estest}. By choosing $\delta$ small enough we conclude 
 \begin{align*}
 \mathcal{Y}_3 \leq C\left( \int_{A_R} e^{2nf}|\partial f|^2 \eta^{s-2} + \frac{1}{R^2} \int_{A_R} e^{2(n-1)f}|\partial f|^2 \eta^{s-4} + \varepsilon R^2 \int_{A_R}\mathcal{Q} \eta^{s}  \right)\, . 
 \end{align*}
\end{proof}

\

\section{Proof of the main results}\label{proofs}

\subsection{Proof of Theorem \ref{teo1}}

Let $(M^{3}, \theta, J, g)$ be a $3$-dimensional complete Sasakian manifold with nonnegative Tanaka-Webster scalar curvature and let $u$ be a nonnegative solution to \eqref{eqcrit1} in $M$. By Bony's maximum principle (see e.g. \cite{Bony}) we have that either $u\equiv 0$ or $u>0$ on $M$. Given $R>0$, let $\eta$ denotes a smooth real cut-off function such that $\eta\equiv 1$ in $B_{R/2}$, $\eta\equiv 0$ in $B_R^c$ and such that $\vert \partial\eta\vert\leq \frac{c}{R}$ in $A_R=B_R\setminus B_{R/2}$. Let $\eps,\beta>0$ small enough and $s>0$ large enough. If $u>0$, then by Lemma \ref{lemma_est}, Lemma  \ref{l-est2} and Lemma \ref{lemma33} we obtain
\begin{align*}
\int_{M} \mathcal{Q}\Psi^{-\beta}\eta^s\leq& C \left( \frac{1}{R^4} \int_{A_R}  \Psi^{-\beta} |\partial f|^2 \eta^{s-4} + \frac{1}{R^2}\int_{A_R}  \Psi^{-\beta} |\partial f|^4 \eta^{s-2} \right)\\
\leq & \frac{C}{R^{6-2\beta}}\vol B_R+C\eps \int_{M} \mathcal{Q}\Psi^{-\beta}\eta^s+ \frac{C}{R^4}\vol B_R\,.
\end{align*}
Using the volume estimate in Theorem \ref{t-comp} we obtain, for every $R>0$,
$$
\int_{B_R} \mathcal{Q}\Psi^{-\beta}\leq C\,. 
$$
Then
$$
\int_M \mathcal{Q}\Psi^{-\beta} < \infty
$$
and, from \eqref{dario},
$$
\int_{B_R}\Psi^{-\beta}|g|^2 \leq C R^2.
$$
From Lemma \ref{lemma0} we obtain
$$
\int_M \mathcal{Q}\Psi^{-\beta}=0\,,
$$
i.e. $\mathcal{Q}\equiv 0$ on $M$. In particular, from \eqref{defQ}, we have
$$
R_{\alpha\bar\alpha}f_{\alpha}f_{\bar\alpha}\equiv 0\quad E_{\alpha\bar\beta}=D_{\alpha\beta}=0\quad G_{\alpha}=D_{\alpha}=E_{\alpha}=0\,.
$$
Since $M$ is $3$-dimensional ($n=1$) we have that $(M^{3}, \theta, J, g)$ is Sasaki-Einstein and
$$
R_{\alpha\bar\alpha}f_{\alpha}f_{\bar\alpha} = R |\partial f|^2 \equiv 0.
$$
By continuity and equation \eqref{eq_g} we have  $\{R\equiv 0\}$ is dense in $M$, thus  $(M^{3}, \theta, J, g)$  is Tanaka-Webster flat (see e.g. \cite{BG}), i.e. its universal cover is isomorphic to the Heinsenberg group $\H^1$. The conclusion follows arguing as in \cite[Section 3]{JL}, see also \cite{CLMR}.

\

\subsection{Proof of Theorem \ref{teo3}}

Let $(M^{2n+1}, \theta, J, g)$, $n\geq 2$, be a $(2n+1)$-dimensional complete Sasakian manifold with  the Tanaka-Webster curvature satisfying
$$
\ricc_H(X,X)\geq 0 \quad \text{for all horizontal vector field }\,X.
$$
and
\begin{equation}\label{vol_ball}
\mathrm{Vol}B_R \leq C R^{2n+2} \, , 
\end{equation}
for some $C>0$ and for all $R>0$ large enough. Let $u$ be a nonnegative solution to \eqref{eqcritn} in $M$ such that 
\begin{equation}\label{energy}
\int_{B_R}u^{\frac{2(n+1)}{n}} \leq C R^{\sigma}\,, 
\end{equation}
for every $R>0$ large enough and for some $\sigma<2$ if $n=2$, $\sigma=2$ if $n\geq 3$. By the maximum principle (see e.g. \cite{Bony}) we have that either $u\equiv 0$ or $u>0$ on $M$. Given $R>0$, let $\eta$ denotes a smooth real cut-off function such that $\eta\equiv 1$ in $B_{R/2}$, $\eta\equiv 0$ in $B_R^c$ and such that $\vert \partial\eta\vert\leq \frac{c}{R}$ in $A_R=B_R\setminus B_{R/2}$. Let $\eps,\beta>0$ small enough and $s>0$ large enough. We split the proof in two cases.

\

\subsection*{Case 1 ($\mathbf{n\geq 3}$):} if $u>0$, then by Lemma \ref{lemma_est} with $\beta=0$ and Lemma \ref{l-estn} we obtain
\begin{align*}
\int_{M} \mathcal{Q}\eta^s\leq& C \left( \frac{1}{R^4} \int_{A_R} e^{2(n-1)f}  |\partial f|^2 \eta^{s-4} + \frac{1}{R^2}\int_{A_R} e^{2(n-1)f}  |\partial f|^4 \eta^{s-2} \right)\\
\leq & C \left( \frac{1}{R^4} \int_{A_R} e^{2(n-1)f}  |\partial f|^2 \eta^{s-4} + \frac{1}{R^2}\int_{A_R} e^{2nf}  |\partial f|^2 \eta^{s-2}+\eps \int_{M} \mathcal{Q}\eta^s \right)\\
\leq & \frac{C}{R^4}\int_{A_R}e^{2nf}\eta^{s-4}+\frac{C}{R^{2n+4}}\vol B_R+\frac{C}{R^2}\int_{A_R} e^{2(n+1)f}  \eta^{s-2}+C\eps \int_{M} \mathcal{Q}\eta^s\,,
\end{align*}
i.e.
$$
\int_{M} \mathcal{Q}\eta^s\leq \frac{C}{R^4}\int_{A_R}e^{2nf}\eta^{s-4} + \frac{C}{R^{2n+4}}\vol B_R+\frac{C}{R^2}\int_{A_R} e^{2(n+1)f}  \eta^{s-2}.
$$
By \eqref{def_f} and H\"older's inequality, we get
\begin{align*}
\int_{M} \mathcal{Q}\eta^s\leq&\frac{C}{R^4}\int_{B_R}u^2 + \frac{C}{R^{2n+4}}\vol B_R+\frac{C}{R^2}\int_{B_R} u^{\frac{2n+2}{n}} \\
\leq & \frac{C}{R^4}\left(\int_{B_R}u^{\frac{2n+2}{n}}\right)^{\frac{n}{n+1}} \left(\vol B_R\right)^{\frac{1}{n+1}}+ \frac{C}{R^{2n+4}}\vol B_R+\frac{C}{R^2}\int_{B_R} u^{\frac{2n+2}{n}}  \\
\leq & C\left( \frac{1}{R^{\frac{2}{n+1}}} + \frac{1}{R^2} +1\right) \leq C 
\end{align*}
where we used \eqref{vol_ball} and \eqref{energy}. Then
$$
\int_M \mathcal{Q} < \infty
$$
and, from \eqref{dario},
$$
\int_{B_R}e^{2(n-1)f}|g|^2 \leq C R^2.
$$
From Lemma \ref{lemma0} we obtain
$$
\int_M \mathcal{Q}=0\,,
$$
i.e. $\mathcal{Q}\equiv 0$ on $M$. In particular, from \eqref{defQ}, we have
$$
R_{\alpha\bar\beta}f_{\beta}f_{\bar\alpha}\equiv 0\quad E_{\alpha\bar\beta}=D_{\alpha\beta}=0\quad G_{\alpha}=D_{\alpha}=E_{\alpha}=0\,.
$$
The conclusion of the theorem follows from the argument at the end of this section.

\

\subsection*{Case 2 ($\mathbf{n= 2}$):} if $u>0$, then by Lemma \ref{lemma_est} with $0<\beta$ small enough, Lemma \ref{l-est2} and Lemma \ref{l-est22} we obtain
\begin{align*}
\int_{M} \Psi^{-\beta}\mathcal{Q}\eta^s\leq& C \left( \frac{1}{R^4} \int_{A_R} e^{2f}  \Psi^{-\beta}|\partial f|^2 \eta^{s-4} + \frac{1}{R^2}\int_{A_R} e^{2f}  \Psi^{-\beta}|\partial f|^4 \eta^{s-2} \right)\\
\leq & C \left( \frac{1}{R^4} \int_{A_R} e^{2f}\Psi^{-\beta}  |\partial f|^2 \eta^{s-4} + \frac{1}{R^2}\int_{A_R} e^{4f}  \Psi^{-\beta}|\partial f|^2 \eta^{s-2}+\eps \int_{M} \Psi^{-\beta}\mathcal{Q}\eta^s \right)\\
\leq & \frac{C}{R^{8-2\beta}}\vol B_R+\frac{C}{R^2}\int_{A_R} e^{(6-2\beta)f}  \eta^{s-2}+\frac{C}{R^4}\int_{A_R} e^{(4-2\beta)f}  \eta^{s-4}+C\eps \int_{M} \Psi^{-\beta} \mathcal{Q}\eta^s\,,
\end{align*}
i.e.
$$
\int_{M} \Psi^{-\beta} \mathcal{Q}\eta^s\leq \frac{C}{R^{8-2\beta}}\vol B_R+\frac{C}{R^2}\int_{A_R} e^{(6-2\beta)f}  \eta^{s-2}+\frac{C}{R^4}\int_{A_R} e^{(4-2\beta)f}  \eta^{s-4}.
$$
By \eqref{def_f} and H\"older's inequality, we get
\begin{align*}
\int_{M} \Psi^{-\beta}\mathcal{Q}\eta^s\leq&\frac{C}{R^{8-2\beta}}\vol B_R+\frac{C}{R^2}\int_{B_R}u^{3-\beta}+\frac{C}{R^4}\int_{B_R} u^{2-\beta} \\
\leq & \frac{C}{R^{8-2\beta}}\vol B_R+\frac{C}{R^2}\left(\int_{B_R}u^3\right)^{\frac{3-\beta}{3}}\left(\vol B_R\right)^{\frac\beta 3}+\frac{C}{R^4}\left(\int_{B_R} u^3\right)^{\frac{2-\beta}{3}}\left(\vol B_R\right)^{\frac{1+\beta}{3}} \\
\leq & C \left(  \frac{1}{R^{2-2\beta}} + \frac{1}{R^{2-\sigma-\left( 2-\frac{\sigma}{3} \right)\beta}} + \frac{1}{R^{2-\frac{2\sigma}{3}-\left( 2-\frac{\sigma}{3} \right)\beta}}\right) 
\end{align*}
which tends to $0$ as $R\rightarrow\infty$, provided $\beta$ is small enough. Thus we deduce $\mathcal{Q}\equiv 0$ on $M$. In particular, from \eqref{defQ}, we have
\begin{equation}\label{eqf}
R_{\alpha\bar\beta}f_{\beta}f_{\bar\alpha}\equiv 0\quad E_{\alpha\bar\beta}=D_{\alpha\beta}=0\quad G_{\alpha}=D_{\alpha}=E_{\alpha}=0\,.
\end{equation}
As $R_{\alpha\bar\beta}\geq 0$, the first identity implies $R_{\alpha\bar\beta}f_{\beta}=0$.

It remains to show that these identities imply that $M=\mathbb{H}^n$ and thus the result follows as in \cite[Section 3]{JL}.

\

\subsection*{Rigidity:} without loss of generality, we assume that $M$ is simply connected. In the following, we work with a unitary frame $\left\{  T_{\alpha}\right\}
$, i.e. $d\theta\left(  T_{\alpha},\overline{T}_{\beta}\right)  =\delta
_{\alpha\beta}$ to get ride of a factor of $2$ in many places. It is also more
convenient to work with $\phi=ce^{-2f}=cu^{-2/n}$. With the proper choice of a
positive constant $c$, the equation (\ref{eqf}) in terms of the function $\phi$ becomes%
\begin{align*}
 R_{\alpha\overline{\beta}}\phi_{\beta}   & =0, \\
\phi_{\alpha,\beta}  &  = 0, \\
\phi_{\alpha,\overline{\beta}}  &  = \phi^{-1}\phi_{\alpha}\phi_{\overline
{\beta}}+\frac{1}{2}\left(  \frac{1}{2}+\phi^{-1}\left\vert \partial
\phi\right\vert ^{2}+i\phi_{0}\right)  \delta_{\alpha\overline{\beta}},\\
\phi_{0,\alpha}  &  =\frac{i}{2}\phi^{-1}\left(  \frac{1}{2}+\phi
^{-1}\left\vert \partial\phi\right\vert ^{2}-i\phi_{0}\right)  \phi_{\alpha
}.
\end{align*}
Moreover, $\phi\rightarrow\infty$ at $\infty$ by the condition $u\rightarrow 0$ at infinity.

Let $\widetilde{\theta}=\phi^{-1}\theta$. In the following, we continue to
work with a local $\theta$-unitary frame $\{T_{\alpha}\}$ and the Reeb vector
field $T$. Then the Reeb vector field of $\widetilde{\theta}$ is
\[
\widetilde{T}=\phi T+i\left(  \phi^{\gamma}T_{\gamma}-\phi
^{\overline{\gamma}}T_{\overline{\gamma}}\right)  .
\]
For $\widetilde{\theta}$, we work with the adapted frame$\left\{  T_{\alpha
},\widetilde{T}\right\}  $ in which the Hermitian metric is given by
$\widetilde{h}_{\alpha\overline{\beta}}=\phi^{-1}\delta_{\alpha\beta}$. We
calculate the Ricci curvature of $\widetilde{\theta}$ by standard formulas
with $w=-\log\phi$,
\begin{align*}
\widetilde{R}_{\alpha\overline{\beta}}  =& R_{\alpha\overline{\beta}%
}-\left\{  \left(  n+2\right)  \left(  w_{\alpha,\overline{\beta}%
}+w_{\overline{\beta},\alpha}\right)  /2+\left[  \Delta_{b}w/2+\left(
n+1\right)  \left\vert \partial w\right\vert ^{2}\right]  \delta
_{\alpha\overline{\beta}}\right\} \\
=&  R_{\alpha\overline{\beta}}-\left\{  -\frac{\left(  n+2\right)  }{2}\left(
\frac{1}{2}\phi^{-1}+\phi^{-2}\left\vert \partial\phi\right\vert ^{2}\right)
\delta_{\alpha\overline{\beta}}\right. \\
&  \left.  +\left[  -\frac{n}{2}\left(  \frac{1}{2}\phi^{-1}+\phi
^{-2}\left\vert \partial\phi\right\vert ^{2}\right)  +\left(  n+1\right)
\phi^{-2}\left\vert \partial\phi\right\vert ^{2}\right]  \delta_{\alpha
\overline{\beta}}\right\} \\
= & R_{\alpha\overline{\beta}}+\frac{\left(  n+1\right)  }{2}\widetilde
{h}_{\alpha\overline{\beta}},
\end{align*}
and its torsion%
\[
\widetilde{A}_{\alpha\beta}=A_{\alpha\beta}+i\left(  w_{\alpha,\beta
}-w_{\alpha}w_{\beta}\right)  =0.
\]
For a function $H$, we write $H_{\alpha,\beta}^{\ast}$ etc. for covariant
derivatives with respect to $\widetilde{\theta}$.

\begin{proposition}
\bigskip The function $\psi:=\phi^{-1}$ satisfies the following three tensor
equations on $\left(  M,\widetilde{\theta}\right)  $
\begin{align*}
\psi_{\alpha,\beta}^{\ast}  &  =0,\\
\psi_{\alpha,\overline{\beta}}^{\ast}-\psi^{-1}\psi_{\alpha}\psi
_{\overline{\beta}}  &  =\frac{1}{2}\left(  g-\psi\right)  \widetilde
{h}_{\alpha\overline{\beta}},\\
\psi_{\widetilde{0},\alpha}^{\ast}  &  =\frac{i}{2}\overline{g}\psi^{-1}%
\psi_{\alpha},
\end{align*}
where $g=\frac{1}{2}\psi+\psi^{-1}\left\vert \partial\psi\right\vert
_{\widetilde{\theta}}^{2}+i\psi_{\widetilde{0}}$ and the subscript
$\widetilde{0}$ stands for differentiation by $\widetilde{T}$. Moreover,%
\[
\widetilde{R}_{\alpha\overline{\beta}}\psi^{\beta}=\frac{\left(  n+1\right)
}{2}\psi_{\alpha}.
\]

\end{proposition}

\begin{proof}
If we write the Tanaka-Webster connection for $\theta$ and that for
$\widetilde{\theta}$ as
\[
\nabla T_{\alpha}=\omega_{\alpha}^{\beta}\otimes T_{\beta}\text{, and
}\widetilde{\nabla}T_{\alpha}=\widetilde{\omega}_{\alpha}^{\beta}\otimes
T_{\beta},
\]
respectively, then
\begin{align*}
\widetilde{\omega}_{\alpha}^{\beta}  = &  \omega_{\alpha}^{\beta}+w_{\gamma
}\delta_{\alpha}^{\beta}\theta^{\gamma}+\left(  w_{\alpha}\theta^{\beta
}-w^{\beta}\delta_{\alpha\overline{\gamma}}\theta^{\overline{\gamma}}\right)
\\
&  +i\left[  w_{,\alpha}^{\beta}+w_{\alpha}w^{\beta}+w_{\gamma
}w^{\gamma}\delta_{\alpha}^{\beta}\right]  \theta.
\end{align*}
With this, one can verify the following formulas%
\begin{align*}
H_{\alpha,\beta}^{\ast}  &  =H_{\alpha,\beta}+\phi^{-1}\left(  \phi_{\alpha
}H_{\beta}+H_{\alpha}\phi_{\beta}\right)  ,\\
H_{\alpha,\overline{\beta}}^{\ast}  &  =H_{\alpha,\overline{\beta}}-\phi
^{-1}\phi_{\gamma}H_{\gamma}\delta_{\alpha\overline{\beta}}\, ,
\end{align*}
and hence the first two formulas for $\psi$. We do it for the second formula. As%
\begin{align*}
\psi_{\alpha,\overline{\beta}}^{\ast}  &  =-\phi^{-2}\left(  \phi
_{\alpha,\overline{\beta}}^{\ast}-2\phi^{-1}\phi_{\alpha}\phi_{\beta}\right)
\\
&  =-\phi^{-2}\left(  \phi_{\alpha,\overline{\beta}}-\phi^{-1}\left\vert
\partial\phi\right\vert ^{2}\delta_{\alpha\overline{\beta}}-2\phi^{-1}%
\phi_{\alpha}\phi_{\beta}\right) \\
&  =-\phi^{-2}\left(  -\phi^{-1}\phi_{\alpha}\phi_{\overline{\beta}}+\frac
{1}{2}\left(  \frac{1}{2}-\phi^{-1}\left\vert \partial\phi\right\vert
^{2}+i\phi_{0}\right)  \delta_{\alpha\overline{\beta}}\right) \\
&  =\psi^{-1}\psi_{\alpha}\psi_{\beta}+\frac{1}{2}\left(  -\frac{1}{2}%
\psi+\psi^{-1}\left\vert \partial\psi\right\vert _{\widetilde{\theta}}%
^{2}+i\phi\psi_{0}\right)  h_{\alpha\overline{\beta}},\\
&  =\psi^{-1}\psi_{\alpha}\psi_{\beta}+\frac{1}{2}\left(  -\frac{1}{2}%
\psi+\psi^{-1}\left\vert \partial\psi\right\vert _{\widetilde{\theta}}%
^{2}+i\psi_{\widetilde{0}}\right)  h_{\alpha\overline{\beta}}%
\end{align*}
which is exactly the second equation. To see the third equation, we have
\begin{align*}
\psi_{0,\alpha}^{\ast}  &  =T_{\alpha}\widetilde{T}\phi^{-1}\\
&  =T_{\alpha}\left(  -\phi^{-1}\phi_{0}\right) \\
&  =-\phi^{-1}\phi_{0\alpha}+\phi^{-2}\phi_{0}\phi_{\alpha}\\
&  =-\frac{i}{2}\phi^{-2}\left(  \frac{1}{2}+\phi^{-1}\left\vert \partial
\phi\right\vert ^{2}-i\phi_{0}\right)  \phi_{\alpha}+\phi^{-2}\phi_{0}%
\phi_{\alpha}\\
&  =-\frac{i}{2}\phi^{-2}\left(  \frac{1}{2}+\phi^{-1}\left\vert \partial
\phi\right\vert ^{2}+i\phi_{0}\right)  \phi_{\alpha}\\
&  =\frac{i}{2}\left(  \frac{1}{2}\psi+\psi^{-2}\left\vert \partial
\psi\right\vert ^{2}-i\phi\psi_{0}\right)  \psi^{-1}\psi_{\alpha}\\
&  =\frac{i}{2}\left(  \frac{1}{2}\psi+\psi^{-1}\left\vert \partial
\psi\right\vert _{\widetilde{\theta}}^{2}-i\psi_{\widetilde{0}}\right)
\psi^{-1}\psi_{\alpha}.
\end{align*}
The last equation follows from the fact $R_{\alpha\overline{\beta}}\phi
_{\beta}=0$.
\end{proof}

From now on, we work on $\left(  M,\widetilde{\theta}\right)  $ and a
$\widetilde{\theta}$-unitary frame (we drop the $\widetilde{}$ and $\ast$
everywhere). To summarize, the function $\psi$ is nonconstant and satisfies%

\begin{align*}
 R_{\alpha\overline{\beta}}\psi_{\beta} &=\frac{\left(  n+1\right)  }{2}%
\psi_{\alpha}\,, \\
\psi_{\alpha,\beta}  &  =0, \\
\psi_{\alpha,\overline{\beta}}  &  =\psi^{-1}\psi_{\alpha}\psi_{\overline
{\beta}}+\frac{1}{2}\left(  -\frac{1}{2}\psi+\psi^{-1}\left\vert \partial
\psi\right\vert ^{2}+i\psi_{0}\right)  \delta_{\alpha\overline{\beta}%
},\\
\psi_{0,\alpha}  &  =\frac{i}{2}\left(  \frac{1}{2}\psi+\psi
^{-1}\left\vert \partial\psi\right\vert ^{2}-i\psi_{0}\right)
\psi^{-1}\psi_{\alpha}.
\end{align*}
Since $w=\log\psi$, then%
\begin{align*}
R_{\alpha\overline{\beta}}w_{\beta} & =\frac{\left(  n+1\right)  }
{2}w_{\alpha}.\\
w_{\alpha,\beta}  &  =-w_{\alpha}w_{\beta},\\
w_{\alpha,\overline{\beta}}  &  =\frac{1}{2}\left(  -\frac{1}{2}+\left\vert
\partial w\right\vert ^{2}+iw_{0}\right)  \delta_{\alpha
\overline{\beta}},\\
w_{0,\alpha}  &  =-\frac{1}{2}w_{0}w_{\alpha}+\frac{i}{2}\left(
\frac{1}{2}+\left\vert \partial w\right\vert ^{2}\right)  w_{\alpha}.
\end{align*}
The following argument is similar to Section 3 in \cite{W1}. We claim that $u$
is CR pluriharmonic Indeed, when $n\geq2$ \ this follows from the second
equation. When $n=1$, differentiating the second equation and simplifying using
all three yields%
\[
w_{\overline{1},11}=\frac{1}{2}\left(  w_{1,1}w_{\overline{1}}+w_{1}%
w_{\overline{1},1}-iw_{0,1}\right)  =0.
\]
As $A_{11}=0$, it follows that $u$ is CR pluriharmonic by \cite[Proposition
3.4]{Lee}. As $M$ is simply connected, $w$ is the real part of a CR
holomorphic function $w+iv$:
\[
v_{\alpha}=-iw_{\alpha},\, v_{\overline{\beta}}=iw_{\overline
{\beta}}.
\]
We also have
\begin{align*}
iv_{0}\delta_{\alpha\beta}  &  =v_{\alpha,\overline{\beta}%
}-v_{\overline{\beta},\alpha}\\
&  =-iw_{\alpha,\overline{\beta}}-iw_{\overline{\beta},\alpha
}\\
&  =-2iw_{\alpha,\overline{\beta}}-w_{0}\delta_{\alpha\beta}\\
&  =-i\left(  -\frac{1}{2}+\left\vert \partial w\right\vert
^{2}\right)  \delta_{\alpha\overline{\beta}}\, . 
\end{align*}
Thus
\[
v_{0}=\frac{1}{2}-\left\vert \partial w\right\vert ^{2}.
\]
With this we can rewrite the equations satisfied by $w$ as%
\begin{align*}
w_{\alpha,\beta}  &  =-w_{\alpha}w_{\beta},\\
w_{\alpha,\overline{\beta}}  &  =\frac{1}{2}\left(  -v_{0}+i%
w_{0}\right)  \delta_{\alpha\overline{\beta}},\\
w_{0,\alpha}  &  =-\frac{1}{2}w_{0}w_{\alpha}+\frac{i}{2}\left(
1-v_{0}\right)  w_{\alpha}\, . 
\end{align*}
We have $\psi=\left\vert e^{\left(  w+iv\right)  /2}\right\vert ^{2}$. As
$\psi\rightarrow0$ at $\infty$, we have $w\rightarrow-\infty$ at $\infty$ and $e^{\left(  w+iv\right)
/2}\rightarrow0$ at $\infty$. Let 
$$
F=e^{\frac{w}{2}}\cos \left(\tfrac{v}{2}\right)+C\, , 
$$
for some constant $C$ to be determined. By the decay condition, we know that $F$ has a critical
point. Moreover $F$ is nonconstant.

\begin{proposition}
\label{crh}We have%
\begin{align*}
R_{\alpha\overline{\beta}}F_{\beta}  &=\frac{\left(  n+1\right)  }%
{2}F_{\alpha}, \\
F_{\alpha,\beta}  &  =0, \\
F_{\alpha,\overline{\beta}}  &  =\frac{1}{2}\left[  -\left(  e^{w/2}\sin
v/2\right)  _{0}+iF_{0}\right]  \delta_{\alpha\overline{\beta}},\\
F_{0,\alpha}  &  =\frac{i}{2}F_{\alpha},\\
F_{0,0}  &  =-\frac{1}{2}\left(  e^{w/2}\sin v/2\right)  _{0}.
\end{align*}

\end{proposition}

\begin{proof}
These formula are proved by direct calculations. For example, to prove the third
one we first observe as $v_{\alpha}=-iw_{\alpha}$ and $\theta$ is
torsion-free
\[
v_{0,\alpha}=v_{\alpha,0}=-iw_{\alpha,0}.
\]
Then we compute using the third equation for $w$%
\begin{align*}
F_{0,\alpha}  &  =\frac{1}{2}e^{w/2}\left[  \left(  w_{0,\alpha}-\frac{1}%
{2}v_{0}v_{\alpha}+\frac{1}{2}w_{0}w_{\alpha}\right)  \cos\frac{v}{2}-\left(
v_{0,\alpha}+\frac{1}{2}w_{0}v_{\alpha}+\frac{1}{2}v_{0}w_{\alpha}\right)
\sin\frac{v}{2}\right] \\
&  =\frac{1}{2}e^{w/2}\left[  \left(  w_{0,\alpha}+\frac{i}{2}%
v_{0}w_{\alpha}+\frac{1}{2}w_{0}w_{\alpha}\right)  \cos\frac{v}{2}-\left(
-iw_{0,\alpha}-\frac{i}{2}w_{0}w_{\alpha}+\frac{1}{2}%
v_{0}w_{\alpha}\right)  \sin\frac{v}{2}\right] \\
&  =\frac{1}{2}e^{w/2}\left(  \frac{i}{2}w_{\alpha}\cos\frac{v}%
{2}+\frac{1}{2}w_{\alpha}\sin\frac{v}{2}\right) \\
&  =\frac{i}{2}e^{w/2}\left(  \frac{1}{2}w_{\alpha}\cos\frac{v}%
{2}-\frac{1}{2}v_{\alpha}\sin\frac{v}{2}\right) \\
&  =\frac{i}{2}F_{\alpha}.
\end{align*}
The first and second formulas can be proved similarly. To prove the last identity, we differentiate the third one%
\begin{align*}
\frac{i}{2}F_{\alpha,\overline{\beta}}  &  =F_{0,\alpha\overline
{\beta}}\\
&  =F_{0,\overline{\beta}\alpha}+iF_{0,0}\delta_{\alpha\overline
{\beta}}\\
&  =\overline{F_{0,\beta\overline{\alpha}}}+iF_{0,0}\delta
_{\alpha\overline{\beta}}\\
&  =-\frac{i}{2}\overline{F_{\beta,\overline{\alpha}}}+iF_{0,0}\delta_{\alpha\overline{\beta}}\, . 
\end{align*}
Using the second identity we obtain%
\[
F_{0,0}=-\frac{1}{2}\left(  e^{w/2}\sin\frac{v}{2}\right)  _{0}%
\]

\end{proof}
Let $D^{2}F$ denote the Hessian of $F$ with respect to the adapted Riemannian metric
$g_{\widetilde{\theta}}$. By Proposition A.3 in the \cite{W1}, we obtain from
Proposition \ref{crh}%
\begin{equation}
D^{2}F=-\frac{1}{2}\left(  e^{w/2}\sin\frac{v}{2}\right)  _{0}g_{\widetilde
{\theta}}. \label{hf}%
\end{equation}
Let $\chi=\frac{1}{2}\left(  e^{w/2}\sin\frac{v}{2}\right)  _{0}$. Then
$D^{2}F=-\chi g_{\widetilde{\theta}}$. Working with a $\widetilde{\theta}%
$-unitary frame, we differentiate (using the Levi-Civita connection of
$g_{\widetilde{\theta}}$) the above equation to get%
\begin{align*}
-\chi_{i}  &  =F_{ji,j}\\
&  =F_{jj,i}+\widehat{R}_{ijlj}w_{l}\\
&  =\left(  \Delta w\right)  _{i}+\widehat{R}_{il}w_{l}\\
&  =-\left(  2n+1\right)  \chi_{i}+\widehat{R}_{il}w_{l},
\end{align*}
where $\widehat{R}$ is the curvature tensor of $g_{\widetilde{\theta}}$ and
the Latin indices take values $0,1,\cdots n,\overline{1},\cdots,\overline{n}$.
Thus $2n\chi_{i}=\widehat{R}_{il}w_{l}$. More specifically,
\begin{align*}
2n\chi_{\alpha}  &  =\widehat{R}_{\alpha0}F_{0}+\widehat{R}_{\alpha\beta
}F_{\overline{\beta}}+\widehat{R}_{\alpha\overline{\beta}}F_{\beta},\\
2n\chi_{0}  &  =\widehat{R}_{00}F_{0}+\widehat{R}_{0\beta}F_{\overline{\beta}%
}+\widehat{R}_{0\overline{\beta}}F_{\beta}.
\end{align*}
By the formulas relating $\widehat{R}_{ij}$ and the pseudohermitian Ricci
curvature (cf. Proposition A.4 in \cite{W1}), they reduce to%
\[
2n\chi_{\alpha}=R_{\alpha\overline{\beta}}F_{\beta}-\frac{1}{2}F_{\alpha
}\,, \quad 2n\chi_{0}=\frac{n}{2}F_{0}.
\]
In view of the equation $R_{\alpha\overline{\beta}}F_{\beta}=\frac{\left(
n+1\right)  }{2}F_{\alpha}$, we conclude that $\chi-\frac{1}{4}F$ is constant.
By choosing $C$, we can assume $\chi=\frac{1}{4}F$. Therefore 
$$
D^{2}F=-\frac{F}{4}g_{\widetilde{\theta}}\, .
$$
This equation was studied by Obata in \cite{Obata}. If $\Sigma=\left\{
F=\lambda\right\}  $ is a regular let set, it is easy to see that the metric locally splits%
\[
g_{\widetilde{\theta}}=dt^{2}+\left[  B\cos\left(\frac{t}{2}\right)-A\sin\left(\frac{t}
{2}\right)\right]^{2}h_{\text{round}},
\]
for some constants $A,B$, where $h_{\text{round}}$ is the standard metric on $\mathbb{S}^{2n}$, and $F=a\cos\left(\frac{t}%
{2}\right)$ for some constant $a$. In other words, $\left(  M,\widetilde{\theta}\right)  $ is locally
isometric to a piece of $\left(  \mathbb{S}^{2n+1},4g_{\text{round}}\right)  $, where
$g_{\text{round}}$ is the standard metric, and $F\left(  z\right)  =\operatorname{Re}%
\left(  z\cdot\overline{\mu}\right)$ is a linear function. Moreover, by
continuation and the above splittings $g_{\widetilde{\theta}}$ is globally spherical. By routine calculations it is easy to see that $\left(
M,\widetilde{\theta}\right)  $ has constant pseudohermitian curvature and
$\theta$ is pseudohermitian flat. Since it is simply connected and complete,
$\left(  M,\theta\right)  $ is CR\ equivalent to $\left(  \mathbb{H}%
^{n},\Theta\right)$, where $\Theta$ is the standard contact form on the Heisenberg group. This concludes the proof of Theorem \ref{teo3} (and of Theorem \ref{teo2}).

\

\subsection{Proof of Corollary \ref{Sob_gen}} We first recall (see e.g. \cite[Theorem 2.6]{BKim}) that the assumption 
$$
\mathrm{Vol}B_R\geq cR^{2n+2}\, ,
$$
for some $c>0$ and every $R>0$ large enough is equivalent to the following Folland-Stein-Sobolev inequality 
$$
\left(\int_M f^{\frac{2n+2}{n}}\right)^{\frac{2n}{2n+2}}\leq C \int_M |\partial f|^2 \quad \text{ for all } f\in\mathcal{D}^{1,2}(M)\, .
$$
Therefore, from Lemma \ref{Moser}, $u$ tends to zero at infinity and the result follows from Theorem \ref{teo2}.

\

\subsection{Proof of Corollary \ref{Sob}} Assume, by contradiction, that there exists $u\in\mathcal{D}^{1,2}(M)$ which minimizes $\mathcal{S}_\theta(M)$. Without loss of generality we can assume that $u$ is  positive. Since $u$ is a minimizer we have that $S_{\theta}(M)>0$, then the Folland-Stein-Sobolev inequality holds and the contradiction follows from Corollary \ref{Sob_gen}.

\

\

\begin{ackn}
\noindent
The first and the third authors are members of and are partially supported by GNSAGA, Gruppo Nazionale per le Strutture Algebriche, Geometriche e le loro Applicazioni of INdAM. The second author is member of and is partially supported by GNAMPA, Gruppo Nazionale per l'Analisi Matematica, la Probabilit\`a e le loro Applicazioni of INdAM.

Moreover, the first and the third authors are partially supported by the project PRIN 2022 ``Differential-geometric aspects of manifolds via Global Analysis'', the second author is partially supported by the project PRIN 2022 ``Geometric-Analytic Methods for PDEs and Applications''.
\end{ackn}

\

\noindent{\bf Data availability statement}

\noindent Data sharing not applicable to this article as no datasets were generated or analysed during the current study.

\

\noindent {\bf Competing Interests}

\noindent The authors declare no competing interests.

\

\

\

\


\begin{thebibliography}{20}

\bibitem{Aubin} T. Aubin. \emph{Probl\`emes isop\'erim\'etriques et espaces de Sobolev.} J. Differential Geometry {\bf 11} (1976), no. 4, 573--598.


\bibitem{book} A. Agrachev, D. Barilari, U. Boscain.  A comprehensive introduction to sub-Riemannian geometry.  Cambridge Studies in Advanced Mathematics, Cambridge Univ. Press,  xviii + 746 pp. (2019).

\bibitem{agralee} A. Agrachev, P. Lee. Bishop and Laplacian comparison theorems on three-dimensional contact sub-Riemannian manifolds with symmetry. J. Geom. Anal. 25 (2015), no.1, 512-535.

\bibitem{BKim}  F. Baudoind, B. Kim. \emph{Sobolev, Poincare and isoperimetric inequalities for subelliptic diffusion operators satisfying a generalized curvature dimension inequality,} Revista Matematica Iberoamericana, 30, (2014), 1, 109--131.


\bibitem{bau} F. Baudoin, E. Grong, K. Kuwada, A. Thalmaier. \emph{Sub-Laplacian comparison theorems on totally geodesic Riemannian foliations,} Calc. Var. (2019) 58:130.




\bibitem{BP} I. Birindelli, J. Prajapat. \emph{Nonlinear Liouville theorems in the Heisenberg group via the moving plane.} Commun. Partial Differ. Equ. {\bf 24} (1999), 1875--1890.



\bibitem{Bony} J.-M. Bony. \emph{Principe du maximum, in\'{e}galit\'{e} de Harnack et unicit\'{e} du probl\`{e}me de Cauchy pour les op\'{e}rateurs elliptiques d\'{e}g\'{e}n\'{e}r\'{e}s.} Ann. Inst. Fourier (Grenoble) {\bf 19} (1969), 277--304.

\bibitem{BG} C. P. Boyer, K. Galicki. \emph{$3-$Sasakian manifolds.} Surveys Diff. Geom. \textbf{7} (1999) 123--184.


\bibitem{CGS} L. Caffarelli, B. Gidas, J. Spruck. \emph{Asymptotic symmetry and local behavior of semilinear elliptic equations with critical Sobolev growth.} Comm. Pure Appl. Math. {\bf 42} (1989), no. 3, 271--297.

\bibitem{CLMR} G. Catino, Y. Y. Li, D. D. Monticelli, A. Roncoroni. \emph{A Liouville theorem in the Heisenberg group.}. J. Eur. Math. Soc., to appear.

\bibitem{CaMo} G. Catino, D. D. Monticelli. \emph{Semilinear elliptic equations on manifolds with nonnegative Ricci curvature.} J. Eur. Math. Soc., to appear.

\bibitem{CMR} G. Catino, D. D. Monticelli, A. Roncoroni. \emph{On the critical $p-$Laplace equation. } Adv. Math. \textbf{433} (2023), 109331.




\bibitem{ChenLi} W. X. Chen, C. Li. \emph{Classification of solutions of some nonlinear elliptic equations.} Duke Math. J. {\bf 63} (1991), no. 3, 615--622.


\bibitem{CLO} W.X. Chen, C. Li, B. Ou. \emph{Classification of solutions for an integral equation.}Comm. Pure Appl. Math. \textbf{59} (2006), no. 3, 330--343.

\bibitem{ChCh} J.-H. Cheng, H.-L. Chiu. \emph{Positive mass theorem and the CR Yamabe equation on 5-dimensional contact spin manifolds.} Adv. Math. {\bf 404} (2022), 108446.


\bibitem{CMY} J.-H. Cheng, A. Malchiodi, P. Yang. \emph{On the Sobolev quotient of three-dimensional CR manifolds.} Rev. Mat. Iberoam., DOI 10.4171/RMI/1412.

\bibitem{CiFaPo} G. Ciraolo, A. Farina, C. Polvara. \emph{Classification results, rigidity theorems and semilinear PDEs on Riemannian manifolds: a $P-$function approach.} Preprint

\bibitem{CFR} G. Ciraolo, A. Figalli, A. Roncoroni. \emph{Symmetry results for critical anisotropic p-Laplacian equations in convex cones. } Geom. Funct.  Anal. \textbf{30} (2020), 770--803.

\bibitem{flyvet} J. Flynn, J. V\'etois, {\em Liouville-type results for the CR Yamabe equation in the Heisenberg group}, preprint, ArXiv.

\bibitem{FMM} M. Fogagnolo, A. Malchiodi, L. Mazzieri, {\em A note on the critical Laplace Equation and Ricci curvature}, J. Geom. Anal. 33 (6) (2023) 1--17.

\bibitem{FS} G. Folland, E. M. Stein. \emph{Estimates for the $\bar{\partial}_{b}$ complex and analysis on the Heisenberg group.} Commun. Pure Appl. Math. \textbf{27} (1974), 429--522.


\bibitem{Gam1} N. Gamara. \emph{CR Yamabe conjecture - the case $n=1$.} J. Eur. Math. Soc. \textbf{3} (2001), 105--137.

\bibitem{Gam2} N. Gamara, R. Yacoub. \emph{CR Yamabe conjecture - the conformally flat case.} Pac. J. Math. \textbf{201} (2001), 121--175.

\bibitem{GarVas} N. Garofalo, D. Vassilev. \emph{Symmetry properties of positive entire solutions of Yamabe type equations on the groups of Heisenberg type.} Duke Math. J. \textbf{106} (2001), 411--448.

\bibitem{GNN} B. Gidas, W. M. Ni, L. Nirenberg. \emph{Symmetry of positive solutions of nonlinear elliptic equations in $\RR^n$.} Mathematical analysis and applications, Part A, pp. 369--402, Adv. in Math. Suppl. Stud., 7a, Academic Press, New York-London, 1981.

\bibitem{GS} B. Gidas, J. Spruck. \emph{Global and local behavior of positive solutions of nonlinear elliptic equations.} Comm. Pure Appl. Math. {\bf 34} (1981), no. 4, 525--598.

\bibitem{asm} A. Hassannezhad, G. Kokarev. \emph{Sub-Laplacian eigenvalue bounds on sub-Riemannian manifolds.} Ann. Sc. Norm. Super. Pisa Cl. Sci. (5) Vol. XVI (2016), 1049--1092.


\bibitem{JL_CR1} D. Jerison, J.M. Lee. \emph{The Yamabe problem on CR manifolds.} J. Differ. Geom.  \textbf{25} (1987), 167--197.

\bibitem{JL} D. Jerison, J.M. Lee. \emph{Extremals for the Sobolev inequality on the Heisenberg group and the CR Yamabe problem. } J Amer. Math. Soc. \textbf{1} (1988), 1--13.

\bibitem{JL_CR2} D. Jerison, J.M. Lee. \emph{Intrinsic CR normal coordinates and the CR Yamabe problem.} J. Differ. Geom.  \textbf{29} (1989), 303--343.


\bibitem{Lee} J.M. Lee, {\em Psuedo-Einstein Structures on CR Manifolds}. Amer. J. Math. {\bf 110} (1988), 157--178.

\bibitem{LeePar} J.M. Lee, T.H. Parker. \emph{The Yamabe problem. } Bull. Amer. Math. Soc. {\bf 17} (1987), 37--91.

\bibitem{leeli} P.  Lee, C.  Li. Bishop and Laplacian comparison theorems on Sasakian manifolds. Comm. Anal. Geom. 26 (2018), no.4, 915-954. 

\bibitem{LiZhang} Y. Li, L. Zhang. \emph{Liouville-type theorems and Harnack-type inequalities for semilinear elliptic equations.} J. Anal. Math. {\bf 90} (2003), 27--87.

\bibitem{OuMa} X. N. Ma, Q. Ou. \emph{A Liouville theorem for a class semilinear elliptic equations on the Heisenberg group.} Adv. Math. {\bf 413} (2023), 108851.

\bibitem{OuXuMa} X. N. Ma, Q. Ou, T. Wu. \emph{Jerison-Lee identities and semilinear subelliptic equations on CR manifolds.} Preprint.

Differential Integral Equations {\bf 25} (2012), no. 7-8, 601-609.

\bibitem{MurSoa} M. Muratori, N. Soave. \emph{Some rigidity results for Sobolev inequalities and related PDEs on Cartan-Hadamard manifolds.} Ann. Sc. Norm. Super. Pisa Cl. Sci. (5) Vol. XXIV (2023), 751--792.

\bibitem{Obata} M. Obata. \emph{The conjectures on conformal transformations of Riemannian manifolds.} J. Differential Geometry {\bf 6} (1971), 247--258.

\bibitem{Ou} Q. Ou. \emph{On the classification of entire solutions to the critical $p-$Laplace equation. } Preprint.

\bibitem{Peral} I. Peral, Multiplicity of Solutions for the $p-$Laplacian, Lecture Notes at the Second
School on Nonlinear Functional Analysis and Applications to Differential Equations, ICTP,
Trieste, 1997.

\bibitem{Rod} E. Rodemich \emph{The Sobolev inequalities with best possible constants.} Analysis Seminar at California Institute of Technology (1966).

\bibitem{Ron} A. Roncoroni.\emph{ An overview on extremals and critical points of the Sobolev inequality in convex cones. } Rend. Lincei Mat. Appl. {\bf 33} (2022), 967--995.


\bibitem{Serrin} J. Serrin. \emph{Local behaviour of solutions of quasi-linear equations.} Acta Math. \textbf{111} (1964), 247--302.

\bibitem{SZ} J. Serrin, H. Zou. \emph{Cauchy-Liouville and universal boundedness theorems for quasilinear elliptic equations and inequalities.} Acta Math. \textbf{1} (189) (2002) 79--142.

\bibitem{Sci} B. Sciunzi. \emph{Classification of positive $\mathcal{D}^{1,p}(\mathbb{R}^n)-$solutions to the critical $p-$Laplace equation in $\mathbb{R}^n$. } Adv. Math. \textbf{291} (2016), 12--23.

\bibitem{St} R.S. Strichartz. \emph{Sub-Riemannian geometry. } J. Diff. Geom. 24 (1986), no.2, 221-263. 

\bibitem{Talenti} G. Talenti. \emph{Best constant in Sobolev inequality.} Ann. Mat. Pura Appl. {\bf 110} (4) (1976), 353--372.

\bibitem{Vet} J. V\'etois. \emph{A priori estimates and application to the symmetry of solutions for critical $p-$Laplace equations. } J. Differential Equations \textbf{260} (2016), 149--161.


\bibitem{Vet_bis} J. V\'etois. \emph{A note on the classification of positive solutions to the critical $p-$Laplace equation in $\mathbb{R}^n$.} Preprint.


\bibitem{W1} X. Wang. \emph{On a remarkable formula of Jerison and Lee in CR geometry.} Math. Res. Lett.  {\bf 22} (2015),  279--299.


\bibitem{W2} X. Wang. \emph{Uniqueness results on a geometric PDE in Riemannian and CR geometry revisited.} Math. Z. \textbf{301} (2022), no.2, 1299--1314.


\end{thebibliography}
\end{document}